\documentclass{birkjour}
 \theoremstyle{definition}
 
 \theoremstyle{remark}

 \numberwithin{equation}{section}

\usepackage{amsmath,amsthm,amsfonts,amssymb}
\usepackage[pdftex]{hyperref}

\usepackage{epstopdf}

\usepackage[final]{pdfpages}
\usepackage{subfig}

\usepackage{rotating}
\usepackage{pdfpages,url}
\usepackage{cancel}
\definecolor{Red}{rgb}{0.7,0,0}
\definecolor{Blue}{rgb}{0,0,0.8}
  \hypersetup{%
    hyperindex = {true},
    colorlinks = {true},
    linktocpage = {true},
    plainpages = {false},
    linkcolor = {Blue},
    citecolor = {Blue},
    urlcolor = {Red},
    pdfstartview = {Fit},
    pdfpagemode = {UseOutlines},
    pdfview = {XYZ null null null}
  }

\usepackage{tikz}
\usetikzlibrary{trees}
\usetikzlibrary{positioning,calc}

\def\fct#1{{\mathop{\rm #1}}}   
\def\Mid{\fct{mid\,}}          

\newcommand{\rpT}{\widehat{\mathcal{T}}^{\downarrow}}
\newcommand{\rpt}[1]{\hat{#1}^{\downarrow}}
\newcommand{\rt}[1]{{#1}^{\downarrow}}
\newcommand{\pt}[1]{\hat{#1}}
\newcommand{\pT}{\widehat{\mathcal{T}}}
\newcommand{\ps}{\mathcal{S}}
\newcommand{\rT}{\mathcal{T}^{\downarrow}}
\newcommand{\T}{\mathcal{T}}

\newcommand{\iV}{\breve{V}}

\newcommand{\comment}[1]{}
\newtheorem{theorem}{Theorem}[section]
\theoremstyle{definition}
\newtheorem{definition}[theorem]{Definition}
\newtheorem{model}[theorem]{Model}
\newtheorem{lemma}[theorem]{Lemma}
\newtheorem{remark}[theorem]{Remark}
\newtheorem{example}[theorem]{Example}
\newtheorem{con}[theorem]{Construction}
\def\calB{\mathcal{B}}
\def\N{\mathbb{N}}

\begin{document}

%
%
%
%
%
%
%
%
%

\title[Distributions on Finite Rooted Binary Trees]{Some Distributions on Finite Rooted Binary Trees}

\author[Cleary, S.]{Sean Cleary}
\address{Department of Mathematics, The City College of New York, NY 10031, USA}
\email{scleary@ccny.cuny.edu}

\author[Fischer, M.]{Mareike Fischer }
\address{Ernst-Moritz-Arndt University of Greifswald, Department for Mathematics and Computer Science, Walther-Rathenau-Str. 47, 17487 Greifswald, Germany}
\email{email@mareikefischer.de}

\thanks{corresponding author: Mareike Fischer}

\author[Griffiths, R.C.]{Robert C. Griffiths}
\address{Department of Statistics, University of Oxford, Oxford OX1 3TG, United Kingdom}
\email{griff@stats.ox.ac.uk}

\author[Sainudiin, R.]{Raazesh Sainudiin}
\address{Department of Mathematics, Uppsala University, Box 480, SE-751 06, Uppsala, Sweden}
\email{raazesh.sainudiin@gmail.com}

\subjclass{Primary 05C05; Secondary 60C05}

\keywords{plane trees, ranked trees, Catalan numbers, dyadic partitions, beta-splitting}

\date{\today}


\begin{abstract}
We introduce some natural families of distributions on rooted binary ranked plane trees with a view toward unifying ideas from various fields, including macroevolution, epidemiology, computational group theory, search algorithms and other fields.  
In the process we introduce the notions of split-exchangeability and plane-invariance of a general Markov splitting model in order to readily obtain probabilities over various equivalence classes of trees that arise in statistics, phylogenetics, epidemiology and group theory.
\end{abstract}

\maketitle

\section{Introduction}

We study some families of distributions on $\rpT_n$, the set of {\em rooted binary ranked plane trees} 
with $n$ unlabeled terminal nodes.  
$\rpT_n$ and their equivalence classes represent various binary tree spaces that are encountered under a myriad of names across several mathematical sciences.  
Our main objective is to introduce existing and novel families of distributions, based on recursively constructive 
randomized tree-generation algorithms, 
whereby the trees are grown from the root node by splitting one of the existing leaf 
nodes according to a probabilistic scheme, to generate distributions on 
$\rpT_{1:n} := \bigcup_{k=1}^n \rpT_k$ and their equivalence classes. 

The plan of the paper is as follows.  
In Section~\ref{S:ClassFRBTrees} we give a brief introduction to four main classes of finite rooted binary trees 
and a recursive randomized construction scheme for the finest class of trees 
that are in bijective correspondence to the permutations.  
In Section~\ref{S:OtherCons} we revisit the class of trees and their recursive construction using other representations. 
In Section~\ref{S:ProbDists} we introduce a nonparametric Markov splitting model on the finest class of ranked plane trees both directly and indirectly through their bijective correspondence with permutations or with dyadic partitions.  
Specific examples of the Markov splitting model are further characterized by split-exchangeability and plane-invariance in order to readily obtain probabilities over various equivalence classes of trees in statistics, phylogenetics, epidemiology and group theory.

\section{Classes of Finite Rooted Binary Trees}\label{S:ClassFRBTrees}

\subsection{Preliminaries}

Recall that a {\em rooted tree}, in the abstract graph-theoretic sense, is a connected acyclic graph with a specific node distinguished as the {\em root}.  
The {\em size} of tree is given by the number of its nodes.  A {\em finite tree} has finitely many nodes. 
In a rooted tree, the {\em outdegree} of a node is the number of its descendants.  
A {\em leaf} is a node of a tree without any descendants.  
Non-leaf nodes are also called {\em internal nodes} and leaf nodes are also called terminal nodes.  
In a {\em rooted binary tree}, every internal node has two descendants.
All trees in this study are finite, rooted and binary and hence we do not explicitly mention this.  
Thus, unless stated otherwise, by a tree we mean a finite rooted binary tree.  
Let $\T_n$ denote the set of all such trees with $n$ leaf nodes.

A tree whose nodes are labeled by distinct elements of a non-empty label set is a {\em labeled tree}.  
We can have a {\em semi-labeled} tree when only a subset of its nodes are labeled by distinct elements.  
Node labels are assigned by a {\em labeling function} from a set of nodes to a set of labels.  
For a formal treatment of semi-labeled trees in a phylogenetic setting see \cite[Dfn.~2.1.1--2]{Steel2003}.
For example, if only the leaf nodes are labeled we get a {\em leaf-labeled tree} and if only the internal nodes 
are labeled we get an {\em internal-labeled tree}.  
An {\em internal-ranking} is a labeling function from the $n-1$ internal nodes of a tree to 
the set of integers in $[n-1]:=\{1,2,\ldots,n-1\}$, which satisfies the following requirements. 
The root node has label or rank $1$, and if $v$ is an interior node which is on 
the path from an interior node $w$ to a leaf node, then the label or rank of $w$  is less than that of $v$. 
A tree together with an internal-ranking gives a {\em ranked tree}.
Let $\rT_n$ denote the set of all such ranked trees with $n$ leaf nodes.
Such trees are also known as {\em increasing trees} \cite{Flajolet}.
 
\comment{
We will formally define various types of phylogenetic trees on $n$ leaves we will observe throughout this paper.

\begin{figure}[htbp]
\begin{center}
\scalebox{0.32}{
\input{FigTrees.pdftex_t}
}
\caption{Example for a ranked labeled tree with leaf label set $\mathfrak{L}=\{1,2,3,4,5\}$, a labeled tree with $\mathfrak{L}=\{1,2,3,4,5\}$, a ranked tree shape and a tree shape (from left to right).\label{FigTrees}}
\end{center}
\end{figure}

\begin{definition} \label{DefTrees}
We define the following trees as in \cite[Sect.~2.4]{Steel2003}.
\begin{itemize}
\item[$(i)$] A \emph{ranked leaf-labeled tree} on $n$ leaves is a rooted binary tree with unique leaf labels from the label set $\mathfrak{L}$. The interior vertices have a total order $<$ assigned, which satisfies the following requirements. The root is the minimum in this order, and if $v$ is an interior vertex $v$ which is on the path from an interior vertex $w$ to a leaf, then $w<v$. Then, the root of the tree is given rank $1$, the second smallest element in this total order has rank $2$, etc.
\item[$(ii)$] A \emph{labeled tree} on $n$ leaves is a ranked labeled tree where the total order with the ranks are omitted.
\item[$(iii)$] A \emph{ranked tree shape} on $n$ leaves is a ranked labeled tree where the leaf labels are omitted.
\item[$(iv)$] A \emph{tree shape} on $n$ leaves is a labeled tree where the leaf labels are omitted.
\end{itemize}
\end{definition}
See \autoref{FigTrees} for an example.
}

\begin{figure}[hbpt]
\begin{tikzpicture}[level distance=0.7cm,
level 1/.style={sibling distance=1.5cm},
level 2/.style={sibling distance=1cm},
level 3/.style={sibling distance=0.5cm}]
\tikzstyle{every node}=[circle,draw]
\node (Root) [black, fill] {}
    child {
    node {} 
    child { node {} child {node [gray, fill]{} } child { node [gray, fill]{} } }
    child { node [gray, fill]{} }
}
child {
    node [gray, fill]{}
};
\end{tikzpicture}
\begin{tikzpicture}[level distance=0.7cm,
level 1/.style={sibling distance=1.5cm},
level 2/.style={sibling distance=1cm},
level 3/.style={sibling distance=0.5cm}]
\tikzstyle{every node}=[circle,draw]
\node (Root) [black, fill] {}
    child {
    node {} 
    child { node [gray, fill]{} }
    child { node {} child {node [gray, fill]{} } child { node [gray, fill]{} } }
}
child {
    node [gray, fill]{}
};
\end{tikzpicture}
\begin{tikzpicture}[level distance=0.7cm,
level 1/.style={sibling distance=1.5cm},
level 2/.style={sibling distance=1cm},
level 3/.style={sibling distance=0.5cm}]
\tikzstyle{every node}=[circle,draw]
\node (Root) [black, fill] {}
    child {
    node [gray, fill]{} 
}
child {
    node {}
    child { node {} child {node [gray, fill]{} } child { node [gray, fill]{} } }
    child { node [gray, fill]{} }
};
\end{tikzpicture}
\begin{tikzpicture}[level distance=0.7cm,
level 1/.style={sibling distance=1.5cm},
level 2/.style={sibling distance=1cm},
level 3/.style={sibling distance=0.5cm}]
\tikzstyle{every node}=[circle,draw]
\node (Root) [black, fill] {}
    child {
    node [gray, fill]{} 
}
child {
    node {}
    child { node [gray, fill]{} }
    child { node {} child {node [gray, fill]{} } child { node [gray, fill]{} } }
};
\end{tikzpicture}
\caption{Four distinct plane trees that represent the same tree with $4$ leaf nodes.  
In each tree, the root node is solid black, all other internal nodes are white with black boundary and leaf nodes are solid gray.\label{F:4RootedPlaneBinTrees}}
\end{figure}
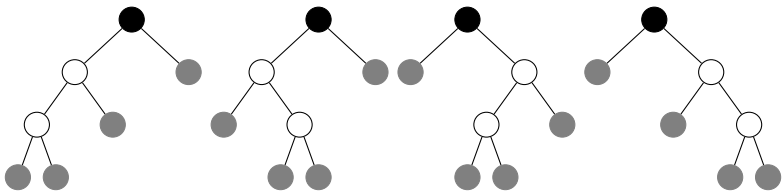

By superimposing additional structure on graph-theoretic trees we obtain trees that are known as plane, planar, oriented or ordered trees.  
A {\em plane tree} is defined as a tree in which subtrees dangling from a common node are ordered between themselves and represented from left to right in order.  
In addition to labels, nodes may also have {\em addresses} to encode their planar embedding or orientation.
It is convenient to assign addresses to the nodes of a plane tree to encode their planar embedding using strings 
formed by concatenation of {\sf L}'s and {\sf R}'s in the obvious manner as shown in 
Figure~\ref{F:NodeNamesRpbTree}.  
Note that all four plane trees in Figure~\ref{F:4RootedPlaneBinTrees} are subtrees of the tree in 
Figure~\ref{F:NodeNamesRpbTree} with addressed nodes.  
The addresses of the nodes of plane trees are clear from the planar drawing and usually not shown as in 
Figure~\ref{F:4RootedPlaneBinTrees}.  
A finite rooted binary tree that is plane is called a {\em plane tree} in this work.  
Figure~\ref{F:4RootedPlaneBinTrees} shows four distinct plane trees that represent the same (non-plane) tree.  
They are known as {\em plane binary trees} in enumerative combinatorics \cite[Ex.~6.19(d), p.~220]{Stanley1999}, {\em finite, rooted binary trees} in geometric group theory \cite[Ch.~10]{Meier2008}, or {\em binary search trees} in computer science \cite{Mahmoud1992}.  
They are less well known in evolutionary biology and may be referred to as {\em rooted binary unranked oriented tree shapes} by a natural extension of phylogenetic notions in \cite[Section~2.4]{Steel2003}.   
Let $\pT_n$ denote the set of all such plane trees with $n$ leaf nodes or equivalently with $n-1$ internal nodes.  
 
\begin{figure}[hbpt]
\begin{tikzpicture}[level distance=1.2cm,
level 1/.append style={sibling distance=5.2cm},
level 2/.append style={sibling distance=2.6cm, font=\small},
level 3/.append style={sibling distance=1.3cm, font=\scriptsize}]
\tikzstyle{every node}=[circle,draw]
\node (Root) [black] {$\rho$}
    child {
    node [black]{$\rho \sf{L}$} 
    child { node {$\rho \sf{LL}$} child {node [black]{$\rho \sf{LLL}$} } child { node [black]{$\rho \sf{LLR}$} } }
    child { node {$\rho \sf{LR}$} child {node [black]{$\rho \sf{LRL}$} } child { node [black]{$\rho \sf{LRR}$} } }
}
child {
    node [black]{$\rho \sf{R}$} 
    child { node {$\rho \sf{RL}$} child {node [black]{$\rho \sf{RLL}$} } child { node [black]{$\rho \sf{RLR}$} } }
    child { node {$\rho \sf{RR}$} child {node [black]{$\rho \sf{RRL}$} } child { node [black]{$\rho \sf{RRR}$} } }
};
\end{tikzpicture}
\caption{Addresses of nodes in plane trees.\label{F:NodeNamesRpbTree}}
\end{figure}
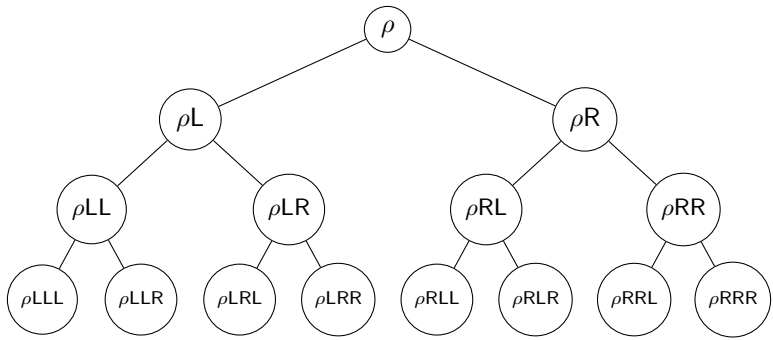

We obtain a {\em ranked plane tree} when an internal ranking is used to assign rank labels from $[n-1]$ 
to the $n-1$ internal (addressed) nodes of a plane tree with $n$ leaf nodes.   
We emphasize that the integer rank label associated with each node of a ranked plane tree is in addition to the node's address given by a string of {\sf L}'s and {\sf R}'s.  
Let $\rpT_n$ denote the set of all ranked plane trees with $n$ leaf nodes.  

\subsection{Ranked Plane Trees}

We can see the rank labels or ranks of a ranked plane tree $\rpt{t}_n \in \rpT_n$ as 
the result of a recursive splitting process of Construction~\ref{R:PlaneBinaryTree} where at the $k$-th step an 
unlabeled leaf node of $\rpt{t}_k \in \rpT_k$ is split and thus transformed into an internal node 
with rank $k$ in order to produce a ranked plane tree $\rpt{t}_{k+1} \in \rpT_{k+1}$ with 
$k$ internal nodes (that have been split) and $k+1$ unlabeled leaf nodes.  
Thus, the ranks encode the splitting order and thereby give the entire history of the process that recursively 
created the tree $\rpt{t}_n \in \rpT_n$ after $n-1$ recursive splits starting from the root node in $\rpT_1$.  

One can assign a random variable $I_k$ over the current set of $k$ leaf nodes (indexed by their addresses) of a 
ranked plane tree and choose a random leaf node according to $I_k$ for the next split.  
This can be used to recursively generate random trees in $\rpT_{1:n} := \bigcup_{k=1}^n \rpT_k$ with up to $n$ 
leaves starting from the root node 
on the basis of the random variables $I_{1:n-1} := \{I_1, I_2,\ldots,I_{n-1}\}$ as follows.

\begin{con}[Ranked Plane Trees]\label{R:PlaneBinaryTree} Consider the following process: 
\begin{itemize}
\item {\sf Initialize}: 
\begin{itemize}
\item $i \gets 1$, set counter 
\item let $\rpt{t}_1=(V,E)=\left(\{\rho\},\{\}\right)$ be a ranked plane tree which consists only of a single root node $\rho$ in $V$, which we consider an unlabeled leaf node, and no edges in $E$.  
\end{itemize}
\item {\sf Randomize}: 
Choose a leaf node $\ell$ of $\rpt{t}_i$ according to the random variable $I_i$ that may depend on $\rpt{t}_i$.
\item {\sf Split $\ell$}: 
\begin{itemize}
\item label $\ell$ by rank $i$,  
\item attach two new leaf nodes $\ell\mathsf{L}$ and $\ell\mathsf{R}$ to the left and the right of $\ell$, respectively -- i.e.~introduce nodes $\ell\mathsf{L}$ and $\ell\mathsf{R}$, and edges $(\ell,\ell\mathsf{L})$ and $(\ell,\ell\mathsf{R})$ to $\rpt{t}_i$,
\item $i \gets i+1$, increment counter 
\end{itemize}
\item {\sf Repeat}: Go to the {\sf Randomize} step if $i$, the number of leaf nodes in $\rpt{t}_i$, is less than a given $n \in \mathbb{N}$.
\end{itemize}
\end{con}

The distribution of the random ranked plane tree $\rpt{t}_n$ produced by the recursive splitting process in 
Construction~\ref{R:PlaneBinaryTree} is determined by $I_{0:n-1}$.  
See Figure~\ref{F:SplitingRecursivelyThrice}(A) for the Hasse diagram on $\rpT_{1:4}$ when $n = 4$.  
By ignoring planarity but not the ranks in $\rpT_{1:n}$ we get random ranked trees (Figure~\ref{F:SplitingRecursivelyThrice}(B)). 
Ignoring the ranks (internal node labels) in $\rpT_{1:n}$ gives random plane trees (Figure~\ref{F:SplitingRecursivelyThrice}(C)).  
Finally by ignoring planarity in addition to the ranks in $\rpT_{1:n}$ we get random trees (Figure~\ref{F:SplitingRecursivelyThrice}(D)).

\begin{figure}[hbpt]
\centering
\begin{tabular}{cccc}
{\subfloat[ranked plane trees]{\fbox{\resizebox{0.64\textwidth}{!}{\begin{tikzpicture}[remember picture,
inner/.append style={},
outer/.append style={level distance=2cm, sibling distance=0.5cm}
]
\node[outer,draw=green,label={[xshift=-0.2cm]{\huge $321$}}] (T3321){
\begin{tikzpicture}[level distance=1.0cm,
level 1/.style={sibling distance=1.5cm},
level 2/.style={sibling distance=1cm},
level 3/.style={sibling distance=0.75cm}]
\tikzstyle{every node}=[circle,draw]
\node (Root)  {$1$}
    child {
    node {$2$} 
    child { node {$3$} child {node [gray, fill]{} } child { node [gray, fill]{} } }
    child { node [gray, fill]{} }
}
child {
    node [gray, fill]{}
};
\end{tikzpicture}
};

\node[outer,draw=green,right=of T3321,label={[xshift=0.6cm]{\huge $231$}}] (T2331) {
\begin{tikzpicture}[level distance=1.0cm,
level 1/.style={sibling distance=1.5cm},
level 2/.style={sibling distance=1.0cm},
level 3/.style={sibling distance=0.75cm}]
\tikzstyle{every node}=[circle,draw]
\node (Root){$1$}
    child {
    node {$2$} 
    child { node [gray,fill]{} }
    child { node {$3$} child {node [gray, fill]{} } child { node [gray, fill]{} } }
}
child {
    node [gray, fill]{}
};
\end{tikzpicture}
};
\node[outer,draw=green,right=of T2331,label={[xshift=0.6cm]{\huge $213$}}] (T2222-1) {
\begin{tikzpicture}[level distance=1.0cm,
level 1/.style={sibling distance=1.5cm},
level 2/.style={sibling distance=1cm},
level 3/.style={sibling distance=0.75cm}]
\tikzstyle{every node}=[circle,draw]
\node (Root)  {$1$}
    child {
    node {$2$} child {node [gray, fill]{} } child { node [gray, fill]{} } 
}
child {
    node {$3$} child {node [gray, fill]{} } child { node [gray, fill]{} } 
};
\end{tikzpicture}
};
\node[outer,draw=green,right=of T2222-1,label={[xshift=-1.0cm]{\huge $312$}}] (T2222-2) {
\begin{tikzpicture}[level distance=1.0cm,
level 1/.style={sibling distance=1.5cm},
level 2/.style={sibling distance=1cm},
level 3/.style={sibling distance=0.75cm}]
\tikzstyle{every node}=[circle,draw]
\node (Root)  {$1$}
    child {
    node {$3$} child {node [gray, fill]{} } child { node [gray, fill]{} } 
}
child {
    node {$2$} child {node [gray, fill]{} } child { node [gray, fill]{} } 
};
\end{tikzpicture}
};
\node[outer,draw=green,right=of T2222-2,label={[xshift=0.6cm]{\huge $132$}}] (T1332) {
\begin{tikzpicture}[level distance=1.0cm,
level 1/.style={sibling distance=1.5cm},
level 2/.style={sibling distance=1cm},
level 3/.style={sibling distance=0.75cm}]
\tikzstyle{every node}=[circle,draw]
\node (Root)  {$1$}
    child {
    node [gray, fill]{} 
}
child {
    node {$2$}
    child { node {$3$} child {node [gray, fill]{} } child { node [gray, fill]{} } }
    child { node [gray, fill]{} }
};
\end{tikzpicture}
};
\node[outer,draw=green,right=of T1332,label={[xshift=0.6cm]{\huge $123$}}] (T1233) {
\begin{tikzpicture}[level distance=1.0cm,
level 1/.style={sibling distance=1.5cm},
level 2/.style={sibling distance=1cm},
level 3/.style={sibling distance=0.75cm}]
\tikzstyle{every node}=[circle,draw]
\node (Root)  {$1$}
    child {
    node [gray, fill]{} 
}
child {
    node {$2$}
    child { node [gray, fill]{} }
    child { node {$3$} child {node [gray, fill]{} } child { node [gray, fill]{} } }
};
\end{tikzpicture}
};
\node[outer,draw=green,label={[xshift=-0.2cm]{\huge $21$}},above=of T2331] (T221) {
\begin{tikzpicture}[level distance=1.0cm,
level 1/.style={sibling distance=1.5cm},
level 2/.style={sibling distance=1cm},
level 3/.style={sibling distance=0.75cm}]
\tikzstyle{every node}=[circle,draw]
\node (Root)  {$1$}
    child {
    node {$2$} child {node [gray, fill]{} } child { node [gray, fill]{} } 
}
child {
    node [gray, fill]{}
};
\end{tikzpicture}
};
\node[outer,draw=green,label={[xshift=0.2cm]{\huge $12$}},above=of T1332] (T122) {
\begin{tikzpicture}[level distance=1.0cm,
level 1/.style={sibling distance=1.5cm},
level 2/.style={sibling distance=1cm},
level 3/.style={sibling distance=0.75cm}]
\tikzstyle{every node}=[circle,draw]
\node (Root)  {$1$}
    child {
    node [gray, fill]{}
}
child {
    node {$2$} child {node [gray, fill]{} } child { node [gray, fill]{} } 
};
\end{tikzpicture}
};
\node[outer,draw=green,label={[xshift=-0.2cm]{\huge $1$}},above left = of T122,xshift=-2.0cm] (T11) {
\begin{tikzpicture}[level distance=1.0cm,
level 1/.style={sibling distance=1.5cm},
level 2/.style={sibling distance=1cm},
level 3/.style={sibling distance=0.75cm}]
\tikzstyle{every node}=[circle,draw]
\node (Root)  {$1$}
    child {
    node [gray, fill]{}
}
child {
    node [gray, fill]{}
};
\end{tikzpicture}
};
\node[outer,draw=green,label={[xshift=-0.2cm]{}},above = of T11,xshift=-0.0cm] (T0) {
\begin{tikzpicture}[level distance=1.0cm,
level 1/.style={sibling distance=1.5cm},
level 2/.style={sibling distance=1cm},
level 3/.style={sibling distance=0.75cm}]
\tikzstyle{every node}=[circle,draw]
\node (Root)  [gray, fill]{};
\end{tikzpicture}
};
\draw[black,->] (T0.south) -- (T11.north);
\draw[black,->] (T11.south) -- (T221.north);
\draw[black,->] (T11.south) -- (T122.north);
\draw[black,->] (T221.south) -- (T2222-1.north);
\draw[black,->] (T122.south) -- (T2222-2.north);
\draw[black,->] (T122.south) -- (T1332.north);
\draw[black,->] (T122.south) -- (T1233.north);
\draw[black,->] (T221.south) -- (T3321.north);
\draw[black,->] (T221.south) -- (T2331.north);
\end{tikzpicture}}}}} & 
\subfloat[ranked trees]{\fbox{\resizebox{0.26\textwidth}{!}{\begin{tikzpicture}[remember picture,
inner/.append style={},
outer/.append style={level distance=2cm, sibling distance=0.5cm}
]
\node[outer,draw=green,label={[xshift=-0.4cm]{\Large $321$}}] (T1233) {
\begin{tikzpicture}[level distance=0.7cm,
level 1/.style={sibling distance=1.5cm},
level 2/.style={sibling distance=1cm},
level 3/.style={sibling distance=0.5cm}]
\tikzstyle{every node}=[circle,draw]
\node (Root)  {$1$}
    child {
    node {$2$} 
    child { node {$3$} child {node [gray, fill]{} } child { node [gray, fill]{} } }
    child { node [gray, fill]{} }
}
child {
    node [gray, fill]{}
};
\end{tikzpicture}
};
\node[outer,draw=green,label={[xshift=0.4cm]{\Large $213$}},right=of T1233] (T2222) {
\begin{tikzpicture}[level distance=0.7cm,
level 1/.style={sibling distance=1.5cm},
level 2/.style={sibling distance=1cm},
level 3/.style={sibling distance=0.5cm}]
\tikzstyle{every node}=[circle,draw]
\node (Root)  {$1$}
    child {
    node {$2$} child {node [gray, fill]{} } child { node [gray, fill]{} } 
}
child {
    node {$3$} child {node [gray, fill]{} } child { node [gray, fill]{} } 
};
\end{tikzpicture}
};
\node[outer,draw=green,label={[xshift=-0.4cm]{\Large $21$}},above=of T2222,xshift=-1.5cm] (T122) {
\begin{tikzpicture}[level distance=0.7cm,
level 1/.style={sibling distance=1.5cm},
level 2/.style={sibling distance=1cm},
level 3/.style={sibling distance=0.5cm}]
\tikzstyle{every node}=[circle,draw]
\node (Root)  {$1$}
    child {
    node {$2$} child {node [gray, fill]{} } child { node [gray, fill]{} } 
}
child {
    node [gray, fill]{}
};
\end{tikzpicture}
};
\node[outer,draw=green,label={[xshift=-0.2cm]{\Large $1$}},above = of T122,xshift=-0.0cm] (T11) {
\begin{tikzpicture}[level distance=0.7cm,
level 1/.style={sibling distance=1.5cm},
level 2/.style={sibling distance=1cm},
level 3/.style={sibling distance=0.5cm}]
\tikzstyle{every node}=[circle,draw]
\node (Root)  {$1$}
    child {
    node [gray, fill]{}
}
child {
    node [gray, fill]{}
};
\end{tikzpicture}
};
\node[outer,draw=green,label={[xshift=-0.2cm]{}},above = of T11,xshift=-0.0cm] (T0) {
\begin{tikzpicture}[level distance=0.7cm,
level 1/.style={sibling distance=1.5cm},
level 2/.style={sibling distance=1cm},
level 3/.style={sibling distance=0.5cm}]
\tikzstyle{every node}=[circle,draw]
\node (Root)  [gray, fill]{};
\end{tikzpicture}
};
\draw[black,->] (T0.south) -- (T11.north);
\draw[black,->] (T11.south) -- (T122.north);
\draw[black,->] (T122.south) -- (T2222.north);
\draw[black,->] (T122.south) -- (T1233.north);
\end{tikzpicture}}}} \\
\subfloat[plane trees]{\fbox{\resizebox{0.64\textwidth}{!}{\begin{tikzpicture}[remember picture,
inner/.append style={},
outer/.append style={level distance=2cm, sibling distance=0.5cm}
]
\node[outer,draw=green] (T3321) {
\begin{tikzpicture}[level distance=0.7cm,
level 1/.style={sibling distance=1.5cm},
level 2/.style={sibling distance=1cm},
level 3/.style={sibling distance=0.5cm}]
\tikzstyle{every node}=[circle,draw]
\node (Root)  {}
    child {
    node {} 
    child { node {} child {node [gray, fill]{} } child { node [gray, fill]{} } }
    child { node [gray, fill]{} }
}
child {
    node [gray, fill]{}
};
\end{tikzpicture}
};
\node[outer,draw=green,right=of T3321] (T2331) {
\begin{tikzpicture}[level distance=0.7cm,
level 1/.style={sibling distance=1.5cm},
level 2/.style={sibling distance=1cm},
level 3/.style={sibling distance=0.5cm}]
\tikzstyle{every node}=[circle,draw]
\node (Root)  {}
    child {
    node {} 
    child { node [gray, fill]{} }
    child { node {} child {node [gray, fill]{} } child { node [gray, fill]{} } }
}
child {
    node [gray, fill]{}
};
\end{tikzpicture}
};
\node[outer,draw=green,right=of T2331] (T2222) {
\begin{tikzpicture}[level distance=0.7cm,
level 1/.style={sibling distance=1.5cm},
level 2/.style={sibling distance=1cm},
level 3/.style={sibling distance=0.5cm}]
\tikzstyle{every node}=[circle,draw]
\node (Root)  {}
    child {
    node {} child {node [gray, fill]{} } child { node [gray, fill]{} } 
}
child {
    node {} child {node [gray, fill]{} } child { node [gray, fill]{} } 
};
\end{tikzpicture}
};
\node[outer,draw=green,right=of T2222] (T1332) {
\begin{tikzpicture}[level distance=0.7cm,
level 1/.style={sibling distance=1.5cm},
level 2/.style={sibling distance=1cm},
level 3/.style={sibling distance=0.5cm}]
\tikzstyle{every node}=[circle,draw]
\node (Root)  {}
    child {
    node [gray, fill]{} 
}
child {
    node {}
    child { node {} child {node [gray, fill]{} } child { node [gray, fill]{} } }
    child { node [gray, fill]{} }
};
\end{tikzpicture}
};
\node[outer,draw=green,right=of T1332] (T1233) {
\begin{tikzpicture}[level distance=0.7cm,
level 1/.style={sibling distance=1.5cm},
level 2/.style={sibling distance=1cm},
level 3/.style={sibling distance=0.5cm}]
\tikzstyle{every node}=[circle,draw]
\node (Root)  {}
    child {
    node [gray, fill]{} 
}
child {
    node {}
    child { node [gray, fill]{} }
    child { node {} child {node [gray, fill]{} } child { node [gray, fill]{} } }
};
\end{tikzpicture}
};
\node[outer,draw=green,above=of T2331] (T221) {
\begin{tikzpicture}[level distance=0.7cm,
level 1/.style={sibling distance=1.5cm},
level 2/.style={sibling distance=1cm},
level 3/.style={sibling distance=0.5cm}]
\tikzstyle{every node}=[circle,draw]
\node (Root)  {}
    child {
    node {} child {node [gray, fill]{} } child { node [gray, fill]{} } 
}
child {
    node [gray, fill]{}
};
\end{tikzpicture}
};
\node[outer,draw=green,above=of T1332] (T122) {
\begin{tikzpicture}[level distance=0.7cm,
level 1/.style={sibling distance=1.5cm},
level 2/.style={sibling distance=1cm},
level 3/.style={sibling distance=0.5cm}]
\tikzstyle{every node}=[circle,draw]
\node (Root)  {}
    child {
    node [gray, fill]{}
}
child {
    node {} child {node [gray, fill]{} } child { node [gray, fill]{} } 
};
\end{tikzpicture}
};
\node[outer,draw=green,above left = of T122,xshift=-0.8cm] (T11) {
\begin{tikzpicture}[level distance=0.7cm,
level 1/.style={sibling distance=1.5cm},
level 2/.style={sibling distance=1cm},
level 3/.style={sibling distance=0.5cm}]
\tikzstyle{every node}=[circle,draw]
\node (Root)  {}
    child {
    node [gray, fill]{}
}
child {
    node [gray, fill]{}
};
\end{tikzpicture}
};
\node[outer,draw=green,above = of T11,xshift=-0.0cm] (T0) {
\begin{tikzpicture}[level distance=0.7cm,
level 1/.style={sibling distance=1.5cm},
level 2/.style={sibling distance=1cm},
level 3/.style={sibling distance=0.5cm}]
\tikzstyle{every node}=[circle,draw]
\node (Root)  [gray, fill]{};
\end{tikzpicture}
};
\draw[black,->] (T0.south) -- (T11.north);
\draw[black,->] (T11.south) -- (T221.north);
\draw[black,->] (T11.south) -- (T122.north);
\draw[black,->] (T221.south) -- (T2222.north);
\draw[black,->] (T122.south) -- (T2222.north);
\draw[black,->] (T122.south) -- (T1332.north);
\draw[black,->] (T122.south) -- (T1233.north);
\draw[black,->] (T221.south) -- (T3321.north);
\draw[black,->] (T221.south) -- (T2331.north);
\end{tikzpicture}}}} & 
\subfloat[trees]{\fbox{\resizebox{0.26\textwidth}{!}{\begin{tikzpicture}[remember picture,
inner/.append style={},
outer/.append style={level distance=2cm, sibling distance=0.5cm}
]
\node[outer,draw=green] (T1233) {
\begin{tikzpicture}[level distance=0.7cm,
level 1/.style={sibling distance=1.5cm},
level 2/.style={sibling distance=1cm},
level 3/.style={sibling distance=0.5cm}]
\tikzstyle{every node}=[circle,draw]
\node (Root)  {}
    child {
    node {} 
    child { node {} child {node [gray, fill]{} } child { node [gray, fill]{} } }
    child { node [gray, fill]{} }
}
child {
    node [gray, fill]{}
};
\end{tikzpicture}
};
\node[outer,draw=green,right=of T1233] (T2222) {
\begin{tikzpicture}[level distance=0.7cm,
level 1/.style={sibling distance=1.5cm},
level 2/.style={sibling distance=1cm},
level 3/.style={sibling distance=0.5cm}]
\tikzstyle{every node}=[circle,draw]
\node (Root)  {}
    child {
    node {} child {node [gray, fill]{} } child { node [gray, fill]{} } 
}
child {
    node {} child {node [gray, fill]{} } child { node [gray, fill]{} } 
};
\end{tikzpicture}
};
\node[outer,draw=green,above=of T2222,xshift=-1.5cm] (T122) {
\begin{tikzpicture}[level distance=0.7cm,
level 1/.style={sibling distance=1.5cm},
level 2/.style={sibling distance=1cm},
level 3/.style={sibling distance=0.5cm}]
\tikzstyle{every node}=[circle,draw]
\node (Root)  {}
    child {
    node {} child {node [gray, fill]{} } child { node [gray, fill]{} } 
}
child {
    node [gray, fill]{}
};
\end{tikzpicture}
};
\node[outer,draw=green,above = of T122,xshift=-0.0cm] (T11) {
\begin{tikzpicture}[level distance=0.7cm,
level 1/.style={sibling distance=1.5cm},
level 2/.style={sibling distance=1cm},
level 3/.style={sibling distance=0.5cm}]
\tikzstyle{every node}=[circle,draw]
\node (Root)  {}
    child {
    node [gray, fill]{}
}
child {
    node [gray, fill]{}
};
\end{tikzpicture}
};
\node[outer,draw=green,above = of T11,xshift=-0.0cm] (T0) {
\begin{tikzpicture}[level distance=0.7cm,
level 1/.style={sibling distance=1.5cm},
level 2/.style={sibling distance=1cm},
level 3/.style={sibling distance=0.5cm}]
\tikzstyle{every node}=[circle,draw]
\node (Root)  [gray, fill]{};
\end{tikzpicture}
};
\draw[black,->] (T0.south) -- (T11.north);
\draw[black,->] (T11.south) -- (T122.north);
\draw[black,->] (T122.south) -- (T2222.north);
\draw[black,->] (T122.south) -- (T1233.north);
\end{tikzpicture}}}} \\
\end{tabular}
\caption{Hasse diagram of the recursive splitting process in Construction~\ref{R:PlaneBinaryTree} to generate 
(A) ranked plane trees in $\rpT_{1:4}$, 
(B) ranked (non-plane) trees in $\rT_{1:4}$, 
(C) plane trees (unranked) in $\pT_{1:4}$, 
and (D) trees (unranked non-plane) in $\T_{1:4}$ with up to three splits and four leaf nodes.  
The permutation representations for trees in (A) and (B) are also given.
\label{F:SplitingRecursivelyThrice}}
\end{figure}
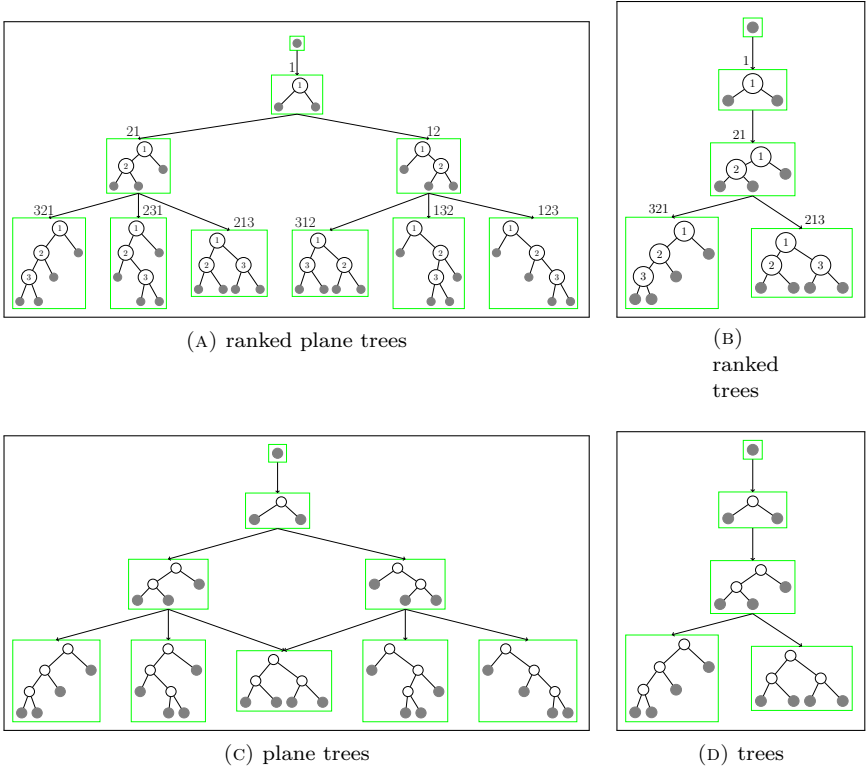

Clearly, every tree $\rpt{t}_n \in \rpT_n$ with $n$ leaves is obtained by 
Construction~\ref{R:PlaneBinaryTree} in a unique way since at the $k$-th splitting step we choose exactly one of 
the available $k$ leaf nodes to split for each $k\in[n-1]$.  
Thus $|\rpT_n|=(n-1)!$.  
There is a simple bijective correspondence between $\rpT_n$ and the $(n-1)!$ permutations of $[n-1]$ 
using the {\em increasing binary tree lifting} (see \cite[Ex.~17, p.~132]{Flajolet} and the references therein).
The bijection, $\rpT_n \ni \rpt{t}_n \leftrightarrow \sigma \in \ps_{n-1}$, 
shown for $n\leq4$ in Figure~\ref{F:SplitingRecursivelyThrice}(A), is given by the following 
Construction~\ref{R:TreeLifting}.
\begin{con}[Tree lifting bijection]\label{R:TreeLifting} Consider the following process: 
\begin{itemize}
\item
Write the permutation $\sigma \in \ps_{n-1}$ as a word $\sigma = \sigma_1\sigma_1\cdots\sigma_{n-1}$. 
\item 
If $\min(\sigma)$ is the minimum letter of $\sigma$, then $\sigma$, as a word, can be decomposed 
into three terms of the form 
$\sigma = \sigma_{\sf L} \cdot \min(\sigma) \cdot \sigma_{\sf R}$, 
with $\sigma_{\sf L},\sigma_{\sf R}$ the words to the left and right of $\min(\sigma)$.  
\item
Then $\rpt{t}_n(\sigma)$, the ranked plane tree corresponding to the given permutation $\sigma$ is obtained by recursively using this decomposition:
\begin{itemize}
\item The empty tree goes with the empty permutation $\varepsilon$. 
\item The root node of the tree $\rpt{t}_n(\sigma)$ gets rank $\min(\sigma)$ with the left and right subtrees 
constructed recursively with $\rpt{t}_n(\sigma_{\sf L})$ and $\rpt{t}_n(\sigma_{\sf R})$, respectively.  
\end{itemize}
\item Conversely, you can get $\sigma$ from a ranked plane tree $\rpt{t}_n$ by simply reading the ranks at the $n-1$ 
internal nodes of $\rpt{t}_n$ in symmetric (in-fix) order.
\end{itemize}
\end{con}

\subsection{Ranked Trees}

The number of ranked trees with $n$ leaves, $|\rT_n|$, is given by the {\em Euler zigzag numbers} \cite{EulerNumbers2013}.  
Ranked trees have been studied in evolutionary biology by Tajima \cite{Tajima1983} as {\em evolutionary relationships} among $n$ nucleons and have recently been given a coalescent re-formulation \cite{Sainudiin2015}.  
They are called {\em unlabeled ranked binary dendrograms} in \cite{Murtagh1984} and can be represented by 
a subset of $\ps_{n-1}$ \cite{Foata1971}.  
Recall the {\em increasing binary tree lifting} that gave a bijection between $\ps_{n-1}$ and $\rpT_n$.  
The idea is to choose a {\em standard} permutation to represent each $\rt{t}$. 
We can use a permutation in $\ps_{n-1}$ to construct a ranked tree in $\rT_n$ 
by modifying Construction~\ref{R:TreeLifting} with a non-planar standard form constraint akin to 
\cite[Section~6]{Murtagh1984}.  

Thus, to obtain a ranked tree $\rt{t} \in \rT_n$ that can be drawn in the plane in a unique way from any 
permutation $\pi \in \ps_{n-1}$ just apply the Construction~\ref{R:TreeLifting} with the additional constraints that must be satisfied by the rank labels at the internal nodes:
\begin{enumerate}
\item if only one of the child nodes $v$ of an internal node $u$ is internal, 
then $v$ is drawn to the left of $u$;
\item if both child nodes $v$ and $w$ of an internal node $u$ are internal with $v < w $, 
then $v$ is drawn on the left of $u$ and $w$ on the right.
\end{enumerate}
Thus, we can use the following non-plane and possibly flipped decomposition of $\sigma$ 
$$
\sigma = \sigma_{\overline{\sf L}} \cdot \min(\sigma) \cdot \sigma_{\overline{\sf R}} =  
\begin{cases}
\sigma_{\sf L} \cdot \min(\sigma) \cdot \sigma_{\sf R} & \text{ if } \min(\sigma_{\sf L}) < \min(\sigma_{\sf R}) 
\text{ or } \sigma_{\sf R} = \varepsilon\\
\sigma_{\sf R} \cdot \min(\sigma) \cdot \sigma_{\sf L} & \text{ if } \min(\sigma_{\sf L}) > \min(\sigma_{\sf R}) 
\text{ or } \sigma_{\sf L} = \varepsilon
\end{cases}
$$
as the only modification in Construction~\ref{R:TreeLifting} 
such that the root node of the tree $\rt{t}_n(\sigma)$ gets 
rank $\min(\sigma)$ with the left and right subtrees 
constructed recursively with $\rt{t}_n(\sigma_{\overline{\sf L}})$ and $\rt{t}_n(\sigma_{\overline{\sf R}})$, 
respectively.  
After constructing such a ranked tree $\rt{t}_n$ from a permutation $\sigma \in \ps_{n-1}$ we can obtain the 
representative non-plane permutation (corresponding to an equivalence class in $\ps_{n-1}$ or 
equivalently in $\rpT_n$) by just reading the ranks at the internal nodes of $\rt{t}_n$ according to 
in-fix order as before.  
Let the set of such representative non-plane permutations be $\mathcal{A}_{n-1}$.  
Clearly $\mathcal{A}_{n-1} \subset \ps_{n-1}$ and 
the above modification to Construction~\ref{R:TreeLifting} due to \cite{Foata1971} 
gives a bijective correspondence $\mathcal{A}_{n-1} \ni \sigma \leftrightarrow \rt{t}_n \in \rT_n$.  
The non-plane permutation representation of ranked trees can be used to enumerate $\mathcal{A}_{n-1}$ which is 
equinumerous to $\rT_n$ using the following recursion:
\begin{eqnarray*}
|\mathcal{A}_{n}|=|\rT_{n+1}|=e(n)
&=&\frac{1}{2}\sum_{k=0}^{n-1}{\left( \binom{n-1}{k} e(k) e(n-k-1)\right)}\\
e(0)&=&e(1)=1
\end{eqnarray*}
The proof is identical to that in \cite[p.~196]{Murtagh1984} although ranks are assigned there in decreasing 
order from the root node. 
Thus, for 
$$n=1,2,\ldots,10, \ |\rT_{n}|=1,1,1,2,5,16,61,272,1385,7936,$$ 
respectively \cite{EulerNumbers2013}.  
Elements of $\rT_{n}$ and $\mathcal{A}_{n-1}$ for $n\in \{1,2,3,4\}$ are shown in 
Figure~\ref{F:SplitingRecursivelyThrice}(B).
We remark in passing that $\mathcal{A}_{n-1}$ is the set of Andr\'e permutations of the second kind (in reverse) 
and that the two kinds of alternating permutations of $[n-1]$ and $\mathcal{A}_{n-1}$ or $\rpT_n$ are in bijective 
correspondence \cite{Donaghey1975}.

\comment{\color{red} Looking at new Figure with 5 of the 5 leaved ranked trees, 
the corresponding permutations in are (from left to right)
(4123),(1234),(3412),(1423) and (3124). 
}

Let a {\em cherry node} be an internal node that has two leaf nodes as its children.
If $\gimel(\rt{t}_n)$ be the number of {\em cherry} nodes of $\rt{t}_n$ then 
$$
|\{\rpt{t}_n \in \rpT_n : \rpt{t}_n \mapsto \rt{t}_n \}| = 2^{n-1-\gimel(\rt{t}_n)} \enspace .
$$
Tajima \cite{Tajima1983} shows that there are $2^{n-1-\gimel(\rt{t}_n)}$ ranked planar trees for a given ranked (non-planar) tree $\rt{t}_n$.  
For an intuitive justification of Tajima's result, suppose we want to turn the ranked tree $\rt{t}_n$ into a ranked planar tree.  
Then, for each of the $n-1$ internal nodes of $\rt{t}_n$, there are two choices for the child node that is said to be `left' except if they are both leaf nodes that carry no ranks (i.e., the internal node is a cherry node).  
Thus, all internal nodes except the cherry nodes (a total of $n-1-\gimel(\rt{t}_n)$ nodes) give two possible orderings for their child nodes with ranks.

\subsection{Plane Trees}

Recall the $n$-th Catalan number \cite{Catalan2005}: 
\[
C_n = \frac{1}{n+1}\binom{2n}{n} = \frac{(2n)!}{(n+1)!n!} = \prod_{k=2}^n \frac{n+k}{k} \enspace.
\]
The number of plane trees with $n-1$ internal nodes and $n$ leaf nodes is given by $C_{n-1}$, 
i.e., $|\pT_n| = C_{n-1}$.  
Recall that the number of ranked plane trees with $n$ leaf nodes is $|\rpT_n| = (n-1)!$, 
and this is greater than the number of plane trees, i.e., $|\rpT_n| > |\pT_n|$ for any $n>2$.  
Thus, if one ignores the ranks at the internal nodes of ranked plane trees and considers them only as (unranked) 
plane trees then by the pigeon-hole principle there may be more than one ranked plane tree that corresponds to a plane tree.  
This combinatorics has to be accounted for when obtaining the distribution on plane trees from that over 
ranked plane trees.  
Thus, $\pT_n$, the set of {\em plane trees} with $n$ leaf nodes and without any internal node labels or ranks 
is an equivalence class of $\rpT_n$.  
The next Lemma gives the needed counting argument.
We suppress sub-scripting trees by the number of leaves for simplicity.

\begin{lemma}\label{L:CatCoeff}
Let $\pt{t}$ be a plane tree with $n$ leaf nodes and $n-1$ internal nodes, 
$\iV(\pt{t}):=\{v \in V: deg(v)>1\}$ be the set of internal nodes of $\pt{t}$, 
$\lfloor \pt{t} \rfloor := |\iV(\pt{t})|$ be the number of internal nodes of $\pt{t}$, 
and $\pt{t}(u)$ be the subtree of $\pt{t}$ with root node $u$.  
Then the {\em Catalan coefficient} \cite{CatalanCoeff2012} of $\pt{t}$, that gives the number of 
ranked plane trees in $\rpT_n$ corresponding to the plane tree $\pt{t} \in \pT_n$, is:
\begin{equation}\label{E:CatCoeff}
B(\pt{t}) = \frac{(n-1)!}{\prod\limits_{u \in \iV(\pt{t})}{\lfloor \pt{t}(u)\rfloor}} 
= {\small \frac{\text{$(\#$ of internal nodes of $\pt{t})!$}}
{\prod\limits_{u \in \iV(\pt{t})}\left(\text{$\#$ of internal nodes of $\pt{t}(u)$}\right)}
}\enspace .
\end{equation}
\begin{proof}
Let $L(\pt{t})$ and $R(\pt{t})$ be left and right subtrees of $\pt{t}$.  
Then the number of distinct binary inter-leavings between the interior (or split) nodes of $L(\pt{t}_n)$ and $R(\pt{t}_n)$ is:
\begin{eqnarray*}
\binom{\lfloor L(\pt{t}) \rfloor + \lfloor R(\pt{t}) \rfloor}{\lfloor L(\pt{t}) \rfloor} 
&=& \frac{\left(\lfloor L(\pt{t}) \rfloor + \lfloor R(\pt{t}) \rfloor \right)!}{ \lfloor L(\pt{t}) \rfloor ! \times \lfloor R(\pt{t}) \rfloor!} 
= \frac{\lfloor \pt{t} \rfloor \times \left(\lfloor L(\pt{t})\rfloor +\lfloor R(\pt{t})\rfloor \right)!}{\lfloor \pt{t}\rfloor \times \lfloor L(\pt{t})\rfloor ! \times \lfloor R(\pt{t})\rfloor !}\\
&=& \frac{\lfloor \pt{t}\rfloor !}{\lfloor \pt{t} \rfloor \times \lfloor L(\pt{t}) \rfloor ! \times \lfloor R(\pt{t}_n)\rfloor !}\enspace .
\end{eqnarray*}
And the number of distinct binary inter-leavings between the interior nodes of $L(\pt{t})$ and $R(\pt{t})$ as well as their subtrees and their sub-subtrees and so on gives the Catalan coefficient by the following recursion with cancellations: 
\begin{eqnarray*}
B(\pt{t}) 
&=& \frac{\lfloor \pt{t} \rfloor !}{\lfloor \pt{t} \rfloor \times \lfloor L(\pt{t}) \rfloor ! \times \lfloor R(\pt{t})\rfloor !} \times B(L(\pt{t})) \times B(R(\pt{t})) \\
&=& \frac{\lfloor \pt{t}|!}{\lfloor \pt{t} \rfloor \times \lfloor L(\pt{t}) \rfloor \times \lfloor R(\pt{t}) \rfloor \times \lfloor L(L(\pt{t})) \rfloor \times \lfloor R(L(\pt{t})) \rfloor \times \cdots \times 1}\\
&=& \frac{\lfloor \pt{t}\rfloor !}{\prod\limits_{v \in \pt{t}} \lfloor\pt{t}(v)\rfloor}  
= \frac{(n-1)!}{\prod\limits_{v \in \iV(\pt{t})} \lfloor\pt{t}(v)\rfloor} \enspace .
\end{eqnarray*}
\end{proof}
\end{lemma}
\begin{remark}
Lemma~\ref{L:CatCoeff} can be proved using poset theoretic ideas as in \cite[Ch.~3, Ex.~1.b., p.~312]{Stanley1997}.  
The proof given above closely follows that of \cite[Cor.~4.1]{Dobrow1995} where $B(\pt{t}_n)$ is called the {\em shape functional} of the planar tree $\pt{t}_n$ in the context of uniform permutations on binary search trees.  
Observe that if a rooted binary phylogenetic tree $\tau$ with $n$ leaf nodes (but with the leaf labels ignored) 
is viewed as a plane tree $\pt{t}$ then the Catalan coefficient of $\pt{t}$ is identical to the number of rankings 
of $\tau$ \cite[Prop.~2.3.2]{Steel2003}.  
\end{remark}

%

Thus, Lemma~\ref{L:CatCoeff} consolidates \cite[Cor.~4.1]{Dobrow1995}, \cite[Ch.~3, Ex.~1.b., p.~312]{Stanley1997} 
and \cite[Prop.~2.3.2]{Steel2003}.  
Our nomenclature is motivated by our need of Catalan coefficients to obtain probabilities on tree spaces with 
up to $n$ leaves akin to how binomial coefficients are needed to obtain probabilities on $\{0,1,2\ldots,n\}$.  
To fix ideas we consider an example and some visualizations of the Catalan coefficients next.

\begin{example}
We can compute the Catalan coefficient of the perfectly balanced plane tree with $k=7$ splits and $8$ leaves (all with depth $3$) using \eqref{E:CatCoeff}, as follows:
\[
B \left(\mathrel{\resizebox{0.1\textwidth}{!}{\begin{tikzpicture}[baseline=(current bounding box.west),level distance=4mm] 
  \tikzstyle{level 1}=[sibling distance=8mm] 
  \tikzstyle{level 2}=[sibling distance=4mm] 
  \tikzstyle{level 3}=[sibling distance=2mm] 
  \coordinate
     child { 
       child {child child}
       child {child child} 
     } 
     child { 
       child {child child} 
       child {child child} 
     };
\end{tikzpicture}
}}\right)
= \frac{7!}{7 \times 3 \times 3 \times 1 \times 1 \times 1 \times 1} = \frac{\overset{2}{\cancel{6}} \times 5 \times 4 \times \cancel{3} \times 2}{\cancel{3} \times \cancel{3}} = 80 \enspace .
\]
This appears at frequency $1$ in the third row of Figure~\ref{F:CatCoeff345Freqs678}.  Similarly,
\[
B \left(\mathrel{\resizebox{0.1\textwidth}{!}{\begin{tikzpicture}[baseline=(current bounding box.west),level distance=4mm] 
  \tikzstyle{level 1}=[sibling distance=8mm] 
  \tikzstyle{level 2}=[sibling distance=4mm] 
  \tikzstyle{level 3}=[sibling distance=2mm] 
  \coordinate
     child { 
       child {child child}
       child 
     } 
     child { 
       child 
       child {child child} 
     };
\end{tikzpicture}
}}\right)
= \frac{5!}{5 \times 2 \times 2 \times 1 \times 1} = \frac{\cancel{5} \times \cancel{4} \times {3} \times 2}{\cancel{5} \times \cancel{2} \times \cancel{2}} = 6 \enspace .
\]
Four of the $42$ plane trees with five splits have Catalan coefficient of $6$ as shown in the third stem plot of the top row of Figure~\ref{F:CatCoeff345Freqs678}.
\end{example}

\begin{figure}[htbp]
\begin{center}
\includegraphics[width=0.8\textwidth]{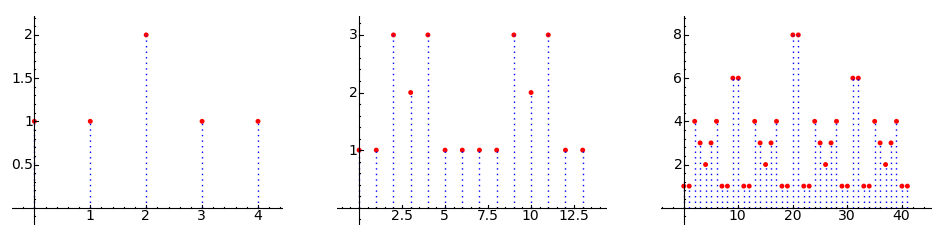}\\
\includegraphics[width=0.8\textwidth]{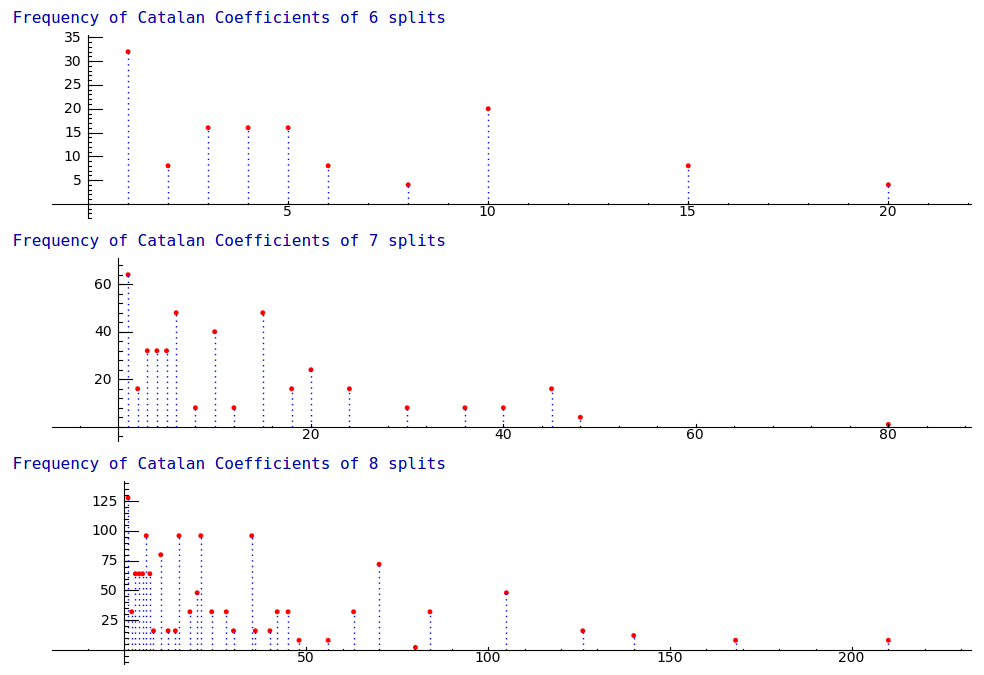}
\end{center}
\caption{Catalan coefficients of plane trees with $3,4,5$ splits (top row) and frequency of Catalan coefficients of plane trees with $6,7,8$ splits.\label{F:CatCoeff345Freqs678}}
\end{figure}

The smallest value of the Catalan coefficient is $1$ (there is only one path leading to a fully left-only branching unbalanced planar tree, for example).  
Thus, $\min (B(\pt{t}_n))=1$.  
Moreover, the number of planar trees with $n-1$ splits that correspond to a non-planar fully unbalanced tree with Catalan coefficient equal to $1$ is $2^{n-2}$.  
This is because, out of the $n-1$ splits in the fully unbalanced tree there are $n-2$ possibilities for planarity (since there are $n-2$ non-cherry internal nodes).  
At the other extreme, the maximum possible value for the Catalan coefficient over trees with $n$ leaves is given by the number of heaps of $n$ elements \cite{MaxHeaps2007}:
\[
1, 1, 1, 2, 3, 8, 20, 80, 210, 896, 3360, 19200, 79200, \ldots \enspace .
\]

Let $Q(\pt{t}_n)=B(\pt{t}_n)/(n-1)!$ be a measure of the balance of the planar tree.  
Clearly, more balanced trees will have a higher value of $Q$ compared to less balanced trees.  
Fill \cite{Fill1996} shows that 
$$-\frac{\ln Q(\pt{t}_n)}{n} \overset{P}{\longrightarrow} \sum_{j=2}^{\infty} 4^{-j}C_j \ln(j) \approxeq 2.03$$ 
if $\pt{t}_n$ is uniformly distributed on $\pT_n$ with $\Pr(\pt{t}_n)=1/C_{n-1}$, and 
$$-\frac{\ln Q(\pt{t}_n)}{n} \overset{P}{\longrightarrow} 2\sum_{k=2}^{\infty}\frac{\ln(k)}{(k+1)(k+2)} \approxeq 1.204$$ 
if the probability of $\pt{t}_n \in \pT_n$ is $Q(\pt{t}_n)$, i.e., induced by the uniform distribution on $\rpT_n$ with $\Pr(\rpt{t}_n)=1/(n-1)!$.  
A limiting Normal law is established in \cite[Thm.~4.1]{Fill1996} for $-\ln(Q(\pt{t}_n))/n$ when $\Pr(\pt{t}_n)=Q(\pt{t}_n)$. 


\subsection{Trees}

Trees in this work are finite rooted and binary without node labels as stated in the Introduction.  
Such trees are called {\em unlabeled non-ranked binary dendrograms} in \cite[Section~5]{Murtagh1984} 
and as {\em rooted binary tree shapes} in \cite[Section~2.4]{Steel2003}.  
Let $\T_n$ denote the set of such trees with $n$ leaf nodes and $n-1$ internal nodes.
The number of such trees as a function of $n$ is given by Wedderburn-Etherington numbers \cite{A001190}:
\[
1, 1, 1, 2, 3, 6, 11, 23, 46, 98, 207, 451, 983, 2179, 4850, 10905, 24631, 56011, \ldots \enspace .
\]

If we generate ranked plane trees in $\rpT_n$ according to Construction~\ref{R:PlaneBinaryTree} 
and ignore the ranks then we can obtain plane trees in $\pT_n$.  
We can further ignore the planarity of trees in $\pT_n$ to obtain trees in $\T_n$.  
From \eqref{E:CatCoeff} we know the number of elements in $\rpT_n$ that map to a given plane tree $\pt{t} \in \pT_n$ when we ignore the ranks.  
Next we find the number of plane trees in $\pT_n$ that map to a given tree $t \in \T_n$.

Recall that $\iV(t)$ is the set of internal nodes of $t$ and $\lfloor t \rfloor :=|\iV(t)|$.  
For each internal node $v$ of $\iV_{t}$, denote by $L(t(v))$ and $R(t(v))$ 
the left and right subtrees of $t$ below $v$, i.e., with $v{\sf L}$ and $v{\sf R}$ as their roots, respectively. 
Let the set of symmetry nodes of $t$ be
$$
S(t) :=\{v\in \iV(t): L(t(v)) \text{ and } R(t(v)) \text{ are isomorphic }\} \text{ and } 
s(t):= |S(t)| \enspace.
$$  

\begin{lemma} Let $t=(V,E) \in \T_n$ with $n \geq 2$.  
Then, $t$ corresponds to $2^{n-1-s(t)}$ plane trees, i.e. 
\begin{equation}\label{E:ShapeCoeff} 
C(t) := |\{\pt{t} \in \pT_n : \pt{t} \mapsto t \}| = 2^{\lfloor t \rfloor -s(t)} = 2^{n-1-s(t)} \enspace .
\end{equation}

\begin{proof} The last equality in \eqref{E:ShapeCoeff} is merely due to the fact that a tree with $n$ leaves has $n-1$ internal nodes, i.e., if $t\in \T_n$ then $\lfloor t \rfloor=n-1$.  
We use induction on the number $n$ of leaves of $t$ to prove the main equality in \eqref{E:ShapeCoeff}.

{\em Initial case $n=2$}: There is only one tree $t$ with two leaves, namely the one consisting of the root node with two attached leaves. 
Clearly, the root is a symmetry node, so $s(t)=1$ and since the root is the only inner node of $t$ we get $\lfloor t \rfloor=1$.  
Also, there is only one plane tree $\pt{t}$ with two leaves.  
It consists of the root node $\rho$ and the two leaf nodes with addresses $\rho{\sf L}$ and $\rho {\sf R}$. 
So altogether, for the number of plane trees we have: $1=2^{\lfloor t \rfloor-s(t)}= 2^{1-1}$.  
This completes the proof for $n=2$.

{\em Inductive case $n\rightarrow n+1$:} 
Let $t$ have $n+1$ leaves and assume that the lemma is already proven for any natural 
number up to and including $n$.  
Let $\rho$ be the root of $t$.  
Let $\rho{\sf L}$ and $\rho{\sf R}$ denote the children of $\rho$ and let $L(t)$ and $R(t)$ denote the subtrees 
rooted at $\rho{\sf L}$ and $\rho{\sf R}$, respectively.  
Let $n_R$, $n_L$ denote the number of leaves in $L(t)$ and $R(t)$, respectively.  
Note that $n_L+n_R=n+1$. 
Moreover, we have $\lfloor t \rfloor = \lfloor L(t) \rfloor + \lfloor R(t) \rfloor + 1$ because of the root.  
As both $n_L$ and $n_R$ are smaller than $n+1$, we know by the inductive assumption that $L(t)$ corresponds to 
$2^{\lfloor L(t) \rfloor - s(L(t))}$ plane trees and $R(t)$ to $2^{\lfloor R(t) \rfloor - s(R(t))}$.  
Now if $L(t)$ and $R(t)$ are isomorphic, $\rho$ is a symmetry node of $t$.  
In this case, $s(t)=s(L(t))+s(R(t))+1$.  
Else, $s(t)=s(L(t))+s(R(t))$.  We consider the two cases next.

\noindent
{\em Symmetric sub-case:} If $\rho$ is a symmetry node, the number of plane trees induced by $t$ is 
just the number of such trees induced by $L(t)$ times that induced by $R(t)$, 
because swapping the left and the right subtrees would not lead to any more distinct plane trees.  
Therefore, we conclude that the number of plane trees induced by $t$ in this case is
\begin{eqnarray*}
2^{\lfloor L(t) \rfloor - s(L(t))} 2^{\lfloor R(t) \rfloor - s(R(t))} 
&=&2^{(\lfloor L(t) \rfloor+\lfloor R(t) \rfloor)-(s(L(t))+s(R(t)))}\\ 
&=&2^{(\lfloor t \rfloor-1)-(s(t)-1)}
=2^{\lfloor t \rfloor-s(t)}\enspace.
\end{eqnarray*}

\noindent
{\em Asymmetric sub-case:} If $\rho$ is not a symmetry node, the number of plane trees induced by $\tau$ is the 
number of such trees induced by $L(t)$ times that induced by $R(t)$ times $2$, 
because the roles of $L(t)$ and $R(t)$ can be swapped about the asymmetric node $\rho$ to obtain two distinct plane trees.  
Therefore, the number of plane trees induced by $t$ in this case is 
\begin{eqnarray*}
2^{\lfloor L(t) \rfloor - s(L(t))} 2^{\lfloor R(t) \rfloor - s(R(t))}  2 
&=&2^{(\lfloor L(t) \rfloor+\lfloor R(t) \rfloor)-(s(L(t))+s(R(t)))+1}\\ 
&=&2^{(\lfloor t \rfloor-1)-s(t)+1}
=2^{\lfloor t \rfloor-s(t)}\enspace.
\end{eqnarray*}
This completes the proof.
\end{proof}
\end{lemma}

More than one plane tree in $\pT_n$ may map to a given tree $t \in \T_n$ when we ignore planarity. 
However, $B(\pt{t})$, the Catalan coefficient of any plane tree  
$\pt{t}$ in $\{ \pt{t} \in \pT_n: \pt{t} \mapsto t \}$, that maps to any given tree $t \in \T_n$, is identical.  
This is because $B(\pt{t})$ in \eqref{E:CatCoeff} only depends on $\prod_{v\in \iV(\pt{t})} \lfloor \pt{t}(v) \rfloor$, 
the product of the number of internal nodes in each subtree with an internal node in $\pt{t}$ as its root, 
a quantity that is preserved when planarity is ignored.  Thus,
\[
B(\pt{t}_n) = B(t_n) = \frac{(n-1)!}{\prod_{v \in \iV} \lfloor t_n(v) \rfloor} \enspace .  
\]
This leads to the next lemma.
\begin{lemma}
The number of ranked planar trees that map to a tree $t_n \in \T_n$ is:
\begin{equation}\label{E:NumRptPerT}
|\{\rpt{t}_n \in \rpT_n : \rpt{t}_n \mapsto t_n\}| = B(t_n) 2^{n-1-s(t_n)}\enspace .
\end{equation}
\end{lemma}

\begin{lemma}
The number of ranked trees in $\rT_n$ corresponding to a given tree $t_n \in \T_n$ is:
\begin{equation}\label{E:NumRtPerT}
|\{\rt{t}_n \in \rT_n : \rt{t}_n \mapsto t_n \}| = \frac{(n-1)!}{\prod\limits_{u \in \iV(\pt{t}_n)}{\lfloor \pt{t}_n(u)\rfloor}} 2^{\gimel(t_n)-s(t_n)}\enspace .
\end{equation}
\end{lemma}

\begin{proof}
Recall that a cherry node is an internal node that has two leaf nodes as its children.
If $\gimel(\rt{t}_n)$ be the number of cherry nodes of $\rt{t}_n$ then
$$
|\{\rpt{t}_n \in \rpT_n : \rpt{t}_n \mapsto \rt{t}_n \}| = 2^{n-1-\gimel(\rt{t}_n)}\enspace .
$$
Also, we know the following two facts:
$$
|\{\rpt{t}_n \in \rpT_n : \rpt{t}_n \mapsto \pt{t}_n \}| = 
B(\pt{t}_n) = \frac{(n-1)!}{\prod\limits_{u \in \iV(\pt{t}_n)}{\lfloor \pt{t}_n(u)\rfloor}} 
$$
and
$$
C(t_n) := |\{\pt{t}_n \in \pT_n : \pt{t}_n \mapsto t_n \}| = 2^{\lfloor t_n \rfloor -s(t_n)} = 2^{n-1-s(t_n)} \enspace ,
$$
where $s(t_n)$ is the size of the set of symmetry nodes of $t_n \in \T_n$:
$$
S(t_n) :=\{v\in \iV(t_n): L(t_n(v)) \text{ and } R(t_n(v)) \text{ are isomorphic }\} \text{ and } 
s(t_n):= |S(t_n)| \enspace.
$$
Due to the invariance of $\gimel$, $s$, $B$ and $C$ to the equivalence classes in $\rpT_n$, $\pT_n$, $\rT_n$ and $\T_n$, we obtain:
\[
|\{\rpt{t}_n \in \rpT_n : \rpt{t}_n \mapsto t_n \}| = B(t_n) C(t_n)  = \frac{(n-1)!}{\prod\limits_{u \in \iV(t_n)}{\lfloor t_n(u)\rfloor}} 2^{n-1-s(t_n)}\enspace .
\]
Finally, we obtain:
\begin{align*}
|\{\rpt{t}_n \in \rpT_n : \rpt{t}_n \mapsto t_n \}| 
&= |\{\rpt{t}_n \in \rpT_n : \rpt{t}_n \mapsto \rt{t}_n \}| \times |\{\rt{t}_n \in \rT_n : \rt{t}_n \mapsto t_n \}|\\
\frac{(n-1)!}{\prod\limits_{u \in \iV(t_n)}{\lfloor t_n(u)\rfloor}} 2^{n-1-s(t_n)} &= 2^{n-1-\gimel(\rt{t}_n)} \times |\{\rt{t}_n \in \rT_n : \rt{t}_n \mapsto t_n \}|\enspace .
\end{align*}
Thus
\[
|\{\rt{t}_n \in \rT_n : \rt{t}_n \mapsto t_n \}| = \frac{(n-1)!}{\prod\limits_{u \in \iV(\pt{t}_n)}{\lfloor \pt{t}_n(u)\rfloor}} 2^{\gimel(t_n)-s(t_n)}\enspace .
\]
\end{proof}

\section{Other Tree Constructions}\label{S:OtherCons}

There are a few representations of the state space for the probabilistic construction of trees.  
The different representations allow different classes of distributions to be defined easily on ranked plane trees.

We can turn any randomized algorithm that generates permutations on $[n-1]$ into one that generates ranked planar trees in $\rpT_{n}$ by simply going from $\sigma \mapsto \rpt{t}$ using Construction~\ref{R:TreeLifting}.  
A simple way to generate random permutations is through the {\em sampling without replacement scheme}, where you start with $n$ balls labelled $1,2\ldots,n$ from an urn, picking one by one, uniformly at random, noting its label and setting it outside the urn in a row.  
Another simple way is through a {\em Knuth Shuffle}, where you start with any permutation (say, the identity permutation), and then go through the positions $1$ through $n-1$, such that for each position $i$ swap the element currently at $i$ with a randomly chosen element from positions $i,i+1,\ldots,n$. 
Although these randomized algorithms over permutations can be transformed using Construction~\ref{R:TreeLifting} into randomized trees, they are not evolutionary as in Construction~\ref{R:PlaneBinaryTree} since the trees are not grown randomly in an incremental manner by splitting one of the existing leaves.

For an evolutionary and incremental construction over permutations, consider a recursive sampling scheme that inserts the $i$-th ball into one of the $i$ gaps between the $i-1$ balls that have been inserted up to step $i$.  
This is equivalent to splitting one of the current leaf nodes of the corresponding ranked plane tree.  
A natural construction of this idea using trees is described next.

A {\em binary search tree} is a rooted planar binary tree, whose internal nodes each store a key (say, a real number) and each internal node has left and right subtrees (see for e.g.~\cite{Mahmoud1992}).  
The tree additionally satisfies the {\em binary search tree property}, whereby the key in each node must be greater than all keys stored in the left subtree, and smaller than all keys stored in right subtree.  
The leaf nodes of the tree contain no key and are usually left unlabelled.  
We are interested in inserting a new key into the tree and growing the tree as summarized in Construction~\ref{R:InsertingRandPerm}.  

\begin{con}[Inserting Random Permutation into Binary Search Tree]\label{R:InsertingRandPerm} Suppose you are given $\sigma=\sigma_1\sigma_2\cdots\sigma_{n-1}$, a random permutation of $[n-1]$.  
First, insert the key $\sigma_1$ into the root node of the binary search tree, a planar tree in $\pT_n$.    
In order to insert the $i$-th node in the tree, its key $\sigma_i$ is first compared with that of the root node, i.e., with $\sigma_1$.  
If its key is less than that of the root, it is then compared with that of the root's left child node. 
If its key is greater than that of the root, it is then compared with that of the root's right child node.  
This process continues, until the new node to be inserted is compared with a sub-terminal node, and then it is added as this node's left or right child, depending on whether its key is greater than or less than the key of the sub-terminal node, respectively.  
\end{con}

For example, the planar tree with three internal nodes and four leaves that is grown by inserting $2,1,3$ is shown in Figure~\ref{F:randpermEx}.  
The first element $2$ is inserted into the root node.  
The second element $1$ is less than $2$ at the root, so it is inserted into the left child node to ensure the binary search tree property.  
Finally, the third element $3$ in the sequence is inserted into the right child node of the root since it is greater than $2$.  

\begin{figure}
{\fbox{\resizebox{0.6\textwidth}{!}{\begin{tikzpicture}[remember picture,
inner/.append style={},
outer/.append style={level distance=2cm, sibling distance=0.5cm}
]
\node[outer,draw=green,label={[xshift=0.4cm]{\Large $213$}},right=of T1233] (T2222) {
\begin{tikzpicture}[level distance=0.7cm,
level 1/.style={sibling distance=1.5cm},
level 2/.style={sibling distance=1cm},
level 3/.style={sibling distance=0.5cm}]
\tikzstyle{every node}=[circle,draw]
\node (Root)  {$2$}
    child {
    node {$1$} child {node [gray, fill]{} } child { node [gray, fill]{} } 
}
child {
    node {$3$} child {node [gray, fill]{} } child { node [gray, fill]{} } 
};
\end{tikzpicture}
};
\node[outer,draw=green,label={[xshift=-0.4cm]{\Large $21$}},left=of T2222,xshift=-1.5cm] (T122) {
\begin{tikzpicture}[level distance=0.7cm,
level 1/.style={sibling distance=1.5cm},
level 2/.style={sibling distance=1cm},
level 3/.style={sibling distance=0.5cm}]
\tikzstyle{every node}=[circle,draw]
\node (Root)  {$2$}
    child {
    node {$1$} child {node [gray, fill]{} } child { node [gray, fill]{} } 
}
child {
    node [gray, fill]{}
};
\end{tikzpicture}
};
\node[outer,draw=green,label={[xshift=-0.2cm]{\Large $2$}},left= of T122,xshift=-0.0cm] (T11) {
\begin{tikzpicture}[level distance=0.7cm,
level 1/.style={sibling distance=1.5cm},
level 2/.style={sibling distance=1cm},
level 3/.style={sibling distance=0.5cm}]
\tikzstyle{every node}=[circle,draw]
\node (Root)  {$2$}
    child {
    node [gray, fill]{}
}
child {
    node [gray, fill]{}
};
\end{tikzpicture}
};
\node[outer,draw=green,label={[xshift=-0.2cm]{}},left= of T11,xshift=-0.0cm] (T0) {
\begin{tikzpicture}[level distance=0.7cm,
level 1/.style={sibling distance=1.5cm},
level 2/.style={sibling distance=1cm},
level 3/.style={sibling distance=0.5cm}]
\tikzstyle{every node}=[circle,draw]
\node (Root)  [gray, fill]{};
\end{tikzpicture}
};
\draw[black,->] (T0.east) -- (T11.west);
\draw[black,->] (T11.east) -- (T122.west);
\draw[black,->] (T122.east) -- (T2222.west);
\end{tikzpicture}}}}
\caption{Binary search tree grown by inserting $2,1,3$.\label{F:randpermEx}}
\end{figure}
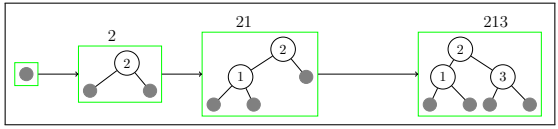



\begin{con}[Dyadic Partition Edges]
Let the set of dyadic fractions be 
$$\mathbb{X} := \{x = (0.b_1b_2\ldots b_{n_x})_2=\sum_{j=1}^{n_x}b_j2^{-j} : b_j \in \{0,1\}, 1 \leq j< n_{x}, b_{n_x}=1, n_x < \infty \} \enspace ,$$ 
where $(0.b_1b_2\ldots b_{n_x})_2$ is the finite binary expansion of the dyadic fraction $x=m/2^n$ with $m,n \in \mathbb{N}$.  
Let $\Mid(a,b)=(b+a)/2$ be the mid-point of $a$ and $b$.  
Let $\check{x}:= 2^{- n_x}$ be the smallest additive constituent of $x$ due to the terminal binary digit $b_{n_x}=1$. 

Our construction gives a sequence of $\mathbb{X}$-valued random variables $(X_0,X_1,\ldots,X_k)$, such that $X_{0}=0$, $X_{1}=1$ and for $k \geq 2$ we obtain $X_k$ from $(X_0,\dots,X_{k-1})$ from a randomly chosen index $I$ for bisecting or splitting as follows:
\[
X_{k} \gets \Mid\left(X_{(I)}, X_{(I+1)}\right), 
\]
where, $X_{(0:k)} := \left(X_{(0)},X_{(1)},\ldots,X_{(i)}, X_{(i+1)},\ldots,X_{(k)}\right)$ is the order statistics of the sequence $X_{0:k} := \left(X_{0},X_{1},\ldots,X_{i}, X_{i+1},\ldots,X_{k}\right)$.  
We also refer to $X_{(0:k)}$ as a partition since it represents the following partition of $[0,1]$ 
$$[X_{(0)}, X_{(1)}) \cup \cdots \cup [X_{(i)}, X_{(i+1)}) \cdots \cup [X_{(k-1)}, X_{(k)}] \enspace .$$
The general probabilistic {\em splitting rule} to obtain $X_k$ from the mid-point of a randomly chosen interval $(X_{(I)}, X_{(I+1)})$ of the partition $X_{(0:k-1)}$ generated by $X_{0:k-1}$ is given by the transition probability matrix $P$ with entries:
\begin{equation}\label{E:GenSplitRule}
P(x_{0:k-1},i) := \Pr\{I=i \, \mid \, X_{0:k-1}=x_{0:k-1}\},  \, i \in \{0,1,\ldots,k-2 \}, k \in {2,3,\ldots}
\end{equation}
\end{con}

\begin{con}[Dyadic Partition Depths] \label{depths}
We can equivalently represent $X_{(0:k)}$ by the width of the successive intervals partitioning $[0,1]$ as follows:
\begin{eqnarray*}
W(X_{(0:k)}) := 
W_{1:k} = \left(W_{1},\ldots, W_{k}\right) =
\left((X_{(1)}-X_{(0)}),\ldots, (X_{(k)}-X_{(k-1)})\right)
\end{eqnarray*}
Note that $\sum_{i=1}^k W_i=1$ and $0 < W_i \leq 1$ and therefore for each $k\geq 1$ we can think of $W_{1:k}=\left( W_1,\ldots,W_k\right)$ as a probability distribution over $k$ outcomes.  
It is convenient to denote $W_{1:k}$ in terms of integer sequences as follows:
\begin{eqnarray*}
Y(X_{(0:k)}) :=
Y_{1:k} = \left(Y_{1},\ldots, Y_{k}\right) =
\left(-\lg(X_{(1)}-X_{(0)}),\ldots, -\lg(X_{(k)}-X_{(k-1)})\right)
\end{eqnarray*}
This is called the depth encoding corresponding to the dyadic partition.  
We can obtain $Y_{1:k}$ from $Y_{1:k-1}$ by choosing $I$ at random according to Equation~\eqref{E:GenSplitRule} and replacing it by two consecutive entries that are deeper by $1$ as follows:
\begin{eqnarray*}
Y_{1} &\gets& Y(X_{(0:1)})=(-\lg(X_{(1)}-X_{(0)}))=(0) \\
Y_{1:k} &\gets& \left(Y_1,Y_2,\ldots,Y_{I}+1,Y_{I}+1, \ldots,Y_{k-1}\right) \enspace .
\end{eqnarray*}
\end{con}

Figure~\ref{Fig:TreeSeq1} depicts the three representations of the binary tree generation process under two splits and Figure~\ref{F:LfDpStateTrDiag4Splits} gives the state transition diagram of the process over dyadic partition depths.

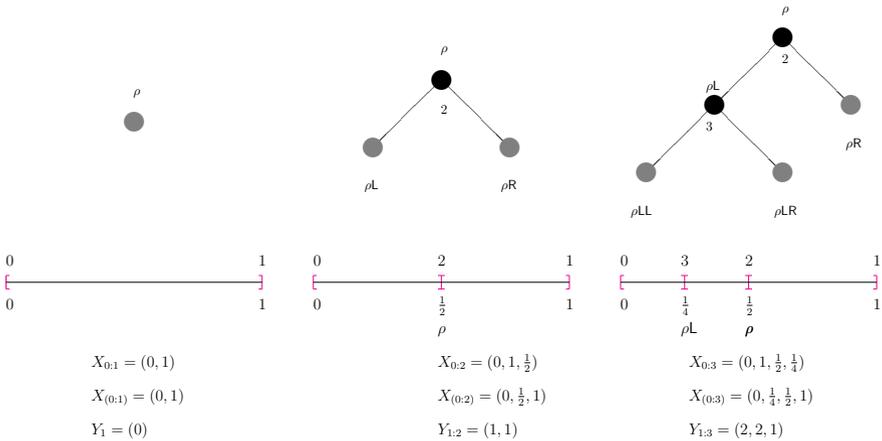
\begin{figure}[htbp]
\begin{center}
\scalebox{0.50}{
\setlength{\unitlength}{9cm}
\begin{picture}(5, 1.3)(0,0.26)
\put(0.375, 1.175){\color{gray}\circle*{0.08}}
\put(0.373,1.255){$\rho$}
\put(0,0.68){\color{magenta}\line(0,1){0.04}}  
\put(0,0.68){\color{magenta}\line(1,0){0.01}}  
\put(0,0.72){\color{magenta}\line(1,0){0.01}}  
\put(0,0.75){\Large{$0$}}
\put(0,0.62){\Large{$0$}}
\put(0.75,0.68){\color{magenta}\line(0,1){0.04}} 
\put(0.74,0.68){\color{magenta}\line(1,0){0.01}}  
\put(0.74,0.72){\color{magenta}\line(1,0){0.01}}  
\put(0.74,0.75){\Large{$1$}}
\put(0.74,0.62){\Large{$1$}}

\put(0,0.70){\line(1,0){0.75}} 
\put(0.25,0.45){\Large{$X_{0:1}=(0,1)$}}
\put(0.25,0.35){\Large{$X_{(0:1)}=(0,1)$}}
\put(0.25,0.25){\Large{$Y_{1}=(0)$}}
\put(1.275, 1.3){\line(-1,-1){0.2}}
\put(1.275, 1.3){\line(1,-1){0.2}}
\put(1.275, 1.3){\circle*{0.08}}
\put(1.273,1.38){$\rho$}
\put(1.273,1.2){$2$}
\put(1.075, 1.1){\color{gray}\circle*{0.08}}
\put(1.05,0.975){$\rho \sf{L}$}
\put(1.475, 1.1){\color{gray}\circle*{0.08}}
\put(1.45,0.975){$\rho \sf{R}$}

\put(0.9,0.68){\color{magenta}\line(0,1){0.04}}  
\put(0.9,0.68){\color{magenta}\line(1,0){0.01}}  
\put(0.9,0.72){\color{magenta}\line(1,0){0.01}}  
\put(0.9,0.75){\Large{$0$}}
\put(0.9,0.62){\Large{$0$}}
\put(1.65,0.68){\color{magenta}\line(0,1){0.04}} 
\put(1.64,0.68){\color{magenta}\line(1,0){0.01}}  
\put(1.64,0.72){\color{magenta}\line(1,0){0.01}}  
\put(1.64,0.75){\Large{$1$}}
\put(1.64,0.62){\Large{$1$}}

\put(1.275,0.68){\color{magenta}\line(0,1){0.04}}  
\put(1.265,0.68){\color{magenta}\line(1,0){0.02}}  
\put(1.265,0.72){\color{magenta}\line(1,0){0.02}}  
\put(1.265,0.75){\Large{$2$}}
\put(1.265,0.62){\Large{$\frac{1}{2}$}}
\put(1.265,0.55){\Large{$\rho$}}
\put(1.265,0.45){\Large{$X_{0:2}=(0,1,\frac{1}{2})$}}
\put(1.265,0.35){\Large{$X_{(0:2)}=(0,\frac{1}{2},1)$}}
\put(1.265,0.25){\Large{$Y_{1:2}=(1,1)$}}

\put(0.9,0.70){\line(1,0){0.75}} 

\put(2.275, 1.425){\circle*{0.08}} 
\put(2.273,1.5){$\rho$}
\put(2.273,1.35){$2$}
\put(2.275, 1.425){\line(-1,-1){0.2}}
\put(2.075, 1.225){\circle*{0.08}} 
\put(2.05,1.27){$\rho \sf{L}$}
\put(2.05,1.15){$3$}
\put(2.075, 1.225){\line(-1,-1){0.2}}
\put(1.875, 1.025){\color{gray}\circle*{0.08}} 
\put(1.83,0.9){$\rho \sf{LL}$}
\put(2.075, 1.225){\line(1,-1){0.2}}
\put(2.275, 1.025){\color{gray}\circle*{0.08}} 
\put(2.25,0.9){$\rho \sf{LR}$}
\put(2.275, 1.425){\line(1,-1){0.2}}
\put(2.475, 1.225){\color{gray}\circle*{0.08}} 
\put(2.46,1.1){$\rho \sf{R}$}

\put(1.8,0.68){\color{magenta}\line(0,1){0.04}}  
\put(1.8,0.68){\color{magenta}\line(1,0){0.01}}  
\put(1.8,0.72){\color{magenta}\line(1,0){0.01}}  
\put(2.55,0.68){\color{magenta}\line(0,1){0.04}} 
\put(2.54,0.68){\color{magenta}\line(1,0){0.01}}  
\put(2.54,0.72){\color{magenta}\line(1,0){0.01}}  
\put(2.54,0.75){\Large{$1$}}
\put(2.54,0.62){\Large{$1$}}

\put(2.175,0.68){\color{magenta}\line(0,1){0.04}}  
\put(2.165,0.68){\color{magenta}\line(1,0){0.02}}  
\put(2.165,0.72){\color{magenta}\line(1,0){0.02}}  
\put(2.165,0.75){\Large{$2$}}
\put(2.165,0.62){\Large{$\frac{1}{2}$}}
\put(2.165,0.55){\Large{$\rho$}}
\put(1.9875,0.68){\color{magenta}\line(0,1){0.04}}  
\put(1.977,0.68){\color{magenta}\line(1,0){0.02}}  
\put(1.977,0.72){\color{magenta}\line(1,0){0.02}}  
\put(1.977,0.75){\Large{$3$}}
\put(1.977,0.62){\Large{$\frac{1}{4}$}}
\put(1.977,0.55){\Large{${\rho \mathsf{L}}$}}

\put(1.8,0.70){\line(1,0){0.75}} 
\put(1.8,0.75){\Large{$0$}}
\put(1.8,0.62){\Large{$0$}}

\put(2.165,0.55){\Large{$\rho$}}
\put(2.0,0.45){\Large{$X_{0:3}=(0,1,\frac{1}{2},\frac{1}{4})$}}
\put(2.0,0.35){\Large{$X_{(0:3)}=(0,\frac{1}{4},\frac{1}{2},1)$}}
\put(2.0,0.25){\Large{$Y_{1:3}=(2,2,1)$}}

\end{picture}
}
\end{center}
\caption{A growing ranked plane tree and the corresponding partition of $[0,1]$.}
\label{Fig:TreeSeq1}
\end{figure}

\comment{
Let $h:=T_1,T_2,\ldots,T_k$ be a sequence of rdipb-trees such that the number of leaves of $T_i$ is $i$ for all $i=1,\ldots,k$ and such that $T_{i+1}$ can be derived from $T_i$ by choosing a leaf $\ell$ of $T_i$ and attaching two new leaves $\ell\mathsf{L}$ and $\ell\mathsf{R}$ to it.  
Then, $h$ is called the {\em splitting path} of $T_k$ and can also be represented by further labelling the internal nodes of $T_k$ according to the order in which they were split.  
When the internal nodes of an rdipb-tree are further labeled by the their splitting order, we refer to it as a {\em split-labeled rdipb-tree}.  
Note that if we label the root node by $1$ for the first split, then such a split-labeled rdipb-tree satisfies the conditions for being a recursive rdipb-tree (ripb-tree).  
Thus, an ripb-tree is equivalent to a dyadic partitioning sequence $X_{0:k}$ and an rdipb-tree (without further labeling of their internal nodes by their splitting order) is equivalent to a partition $X_{(0:k)}$.
}

A specific sequence $x_{0:k}$ represents a sequence of $k-1$ bisections or splits of $[0,1]$ into $k$ intervals and the order statistics $x_{(0:k)}$ represents the corresponding partition formed by the sequence of splits.  
Let $\mathcal{X}_k$ denote the set of all dyadic partitions of size $k$ which is in bijective correspondence with $\pT_k$, the set of plane trees with $k$ leaves.  
Thus, $\# \mathcal{X}_{k} = \# \pT_k = C_{k-1}$, the $(k-1)$-th Catalan number.  
When there are $k$ intervals in $X_{0:k}$ there are $k$ possible splits leading to $k$ choices for $X_{0:k+1}$.  
Thus, there are $(k-1)!$ distinct sequences for $X_{0:k}$ which is in bijective correspondence with $\rpT_k$, the set of ranked plane trees with $k$ leaves.  

\begin{figure}[htbp]
\begin{center}
\includegraphics[width=\textwidth]{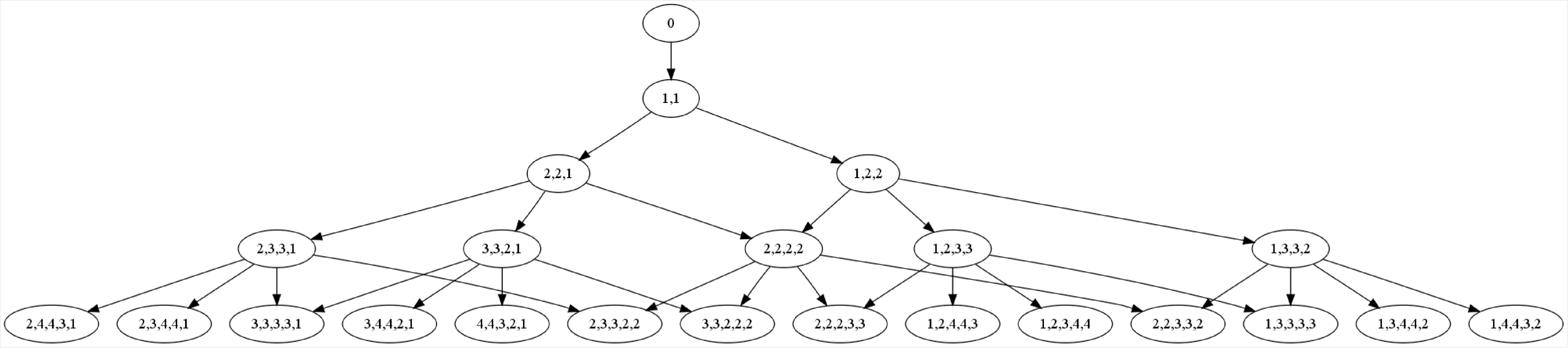}\\
\end{center}
\caption{
State Transition Diagram  with $0$, $1$, $2$, $3$, and $4$ splits.
}
\label{F:LfDpStateTrDiag4Splits}
\end{figure}

\section{Nice properties of some familiar probability models}\label{S:ProbDists}

\begin{model}[General Splitting]\label{M:GLS}
In the most general model, we allow the transition probabilities given in Equation~\eqref{E:GenSplitRule} to possibly depend on the entire history of $X_{0:k-1}$.  
For each level $k \in \{2,3,\ldots\}$, corresponding to $k-1$ splits, the transition matrix $P(x_{0:k-1},i)$ has $(k-1)!$ rows, corresponding to the number of distinct possibilities for $x_{0:k-1}$, and $k-1$ columns, corresponding to the number of intervals or leaves in the partition $x_{(0:k-1)}$ associated with each such sequence $x_{0:k-1}$.  
Since the rows of $P$ must be non-negative and sum to $1$, we can think of each row as being a point in the $(k-1)$-simplex:
\[
\Delta_{k-1} := \{(p_0,p_1,\ldots, p_{k-2}): p_i \geq 0, \forall i, \, \sum_{i=0}^{k-2} p_i = 1 \}
\]  
So each family of such transition probabilities can be thought of as an element of:
\[
\{ \Delta_{k-1}^{(k-1)!} : k \in \{2,3,\ldots \} \}
\]  
that can index the law of a partitioning or tree-building process. 
\end{model}

Besides specifying the largest non-parametric family for the tree-building process, Model~\ref{M:GLS} is too general to provide useful insights.  
We next restrict the construction to satisfy a Markov property on state space $\mathcal{X}$, the set of all dyadic partitions of $[0,1]$.  
We want the probability of $X_k$ given the entire history $X_{0:k-1}$ to only depend on the partition $X_{(0:k-1)}$:
\[
\Pr \{ X_k \mid X_{0:k-1} \} = \Pr \{ X_k \mid X_{(0:k-1)} \} \enspace .
\]

\begin{model}[Markov Splitting]\label{M:MS}
Here we allow the transition probabilities given in Equation~\eqref{E:GenSplitRule} to only depend on the entire history of $X_{0:k-1}$ up to the most recent partition $X_{(0:k-1)}$.  
We can accomplish this by ensuring that the transition probabilities satisfy:
\[
P(x_{0:k-1},i) = \Pr \{ I \, \mid \, X_{0:k-1} \} = \Pr \{ I \, \mid \, X_{(0:k-1)} \} = P(x_{(0:k-1)},i),
\]
where, $X_k = \Mid(X_{(I)},X_{(I+1)})$.
For each level $k \in \{2,3,\ldots\}$, corresponding to $k-1$ splits, the transition matrix $P(x_{0:k-1},i) = P(x_{(0:k-1)},i)$ has only $C_k$ rows, corresponding to the number of distinct possibilities for $x_{(0:k-1)}$, and $k-1$ columns, corresponding to the number of intervals or leaves in the partition $x_{(0:k-1)}$.  
Thus, the construction gives a Markov chain on state space $\mathcal{X}$, the set of all dyadic partitions of $[0,1]$, that satisfies the following Markov property
\[
\Pr \{ X_{(0:k)} \mid X_{(0:1)}, X_{(0:2)},\ldots,X_{(0:k-1)} \} = \Pr \{ X_{(0:k)} \mid X_{(0:k-1)} \}
\]
So each family of such transition probabilities can be thought of as an element of:
\[
\{ \Delta_{k-1}^{C_{k-1}} : k \in \{2,3,\ldots \} \}
\]
that can index the law of a partitioning or tree-building process.  
Finally, instead of the bisection scheme where $X_k = \Mid(X_{(I)},X_{(I+1)})$, we can substitute a more general way of splitting the interval $(X_{(I)},X_{(I+1)})$ into two subintervals.  
For instance, we can sample a point $s$ from a density $g$ rescaled over $(X_{(I)},X_{(I+1)})$, such that $\int_{X_{(I)}}^{X_{(I+1)}}g(x)dx=1$, and use it to split $(X_{(I)},X_{(I+1)})$ into $(X_{(I)},s)$ and $(s,X_{(I+1)})$. 
\end{model}

Next we present some concrete Markov splitting models that are special cases of Model~\ref{M:MS}. 

\begin{model}[Uniform Splitting]\label{M:US}
A concrete example of the conditional random variable $I$ is $\rm{Uniform} \{0,1,\ldots,k-2 \}$ with 
\[
P(x_{0:k-1},i)=P(x_{(0:k-1)},i)=
\begin{cases}
1/(k-1) & \text{if } i \in \{0,1,\ldots,k-2\}, k \in \{2,3,\ldots\} \enspace ,\\
0 & \text{otherwise} \enspace .
\end{cases}
\]  
This corresponds to producing the next split by choosing one of the current intervals or leaves uniformly at random.  
This model assigns uniform probability $1/(k-1)!$ to every ranked planar tree $\rpt{t}_k \in \rpt{T}_k$ with $k-1$ splits and $k$ leaves and is equivalent to the speciation model due to Yule \cite{Yule1924} in phylogenetics and the random permutation model for binary search trees \cite{Dobrow1995}.
\end{model}

\begin{model}[Statistically Equivalent Block or SEB Splitting]\label{M:IfS}
The distribution of $I$ can be given by a probability density function $f$ on $[0,1]$ such that $\int_0^1f(u)du=1$, $f(u) \geq 0$ for every $u \in [0,1]$ and $f(u)=0$ for every $u \notin [0,1]$.   
Under this model we choose the next leaf for splitting according to:
\[
P(x_{0:k-1},i)=P(x_{(0:k-1)},i)=\int_{x_{(i)}}^{x_{(i+1)}}f(u)du, \quad i \in \{0,1,\ldots,k-2\}, k \in \{2,3,\ldots\} .
\]
And having chosen a leaf interval, we split that leaf interval exactly at its mid-point.
For the special case of the continuous random variable on $[0,1]$ with uniform density: $f(u)=1$ if $u \in [0,1]$ and $f(u)=0$ if $u \notin [0,1]$, we have the Uniform SEB Splitting Model where intervals are bisected according to probabilities given by their widths.  
This model is indexed by a density $f$ on $[0,1]$ and produces trees such that its leaf intervals have nearly uniform probability under $f$ since the leaf interval with the most probability under $f$ is immediately bisected.  
This is related to the statistically equivalent blocks rule as a consistent partitioning strategy in density estimation \cite{Gessaman1970}.
\end{model}

\begin{model}[Depth-proportional Splitting]\label{M:DPS}
The distribution of $I$ is obtained by normalizing the depth of each leaf in $x_{(0:k-1)}$ or $y_{1:k}$ as follows:
\[
P(x_{0:k-1},i)=P(x_{(0:k-1)},i)=\frac{-\lg(x_{(i)}-x_{(i-1)})}{\sum_{i=1}^{k-1}-\lg(x_{(i)}-x_{(i-1)})} = \frac{y_i}{\sum_{i=1}^{k-1} y_i}= P(y_{1:k-1},i), 
\]
where, $i \in \{1,\ldots,k-1\}, k \in \{2,3,\ldots\}$.  
Thus, deeper nodes have a higher probability of splitting under this model.
\end{model}

Blum and Fran\c{c}ois \cite{Blum2006} introduced an evolutionary Beta-splitting model based on ideas of Kirkpatrick and Slatkin \cite{Kirkpatrick1993}, and Aldous \cite{Aldous2001}.  
This model is further extended to a biparametric Beta-splitting model for diversification in \cite{SainudiinVeber2016} and for epidemiological transmission in \cite{SainudiinWelch2016}.  
Under this model, a current interval is chosen for bisection with probability given by its width and the point of bisection is drawn from a Beta$(\alpha+1,\beta+1)$ distribution rescaled over the interval being bisected, for some $\alpha>-1$ and $\beta>-1$.  
This way of splitting a leaf interval is in contrast with earlier models where we always bisected the chosen leaf interval at its mid-point.  
The following model is from \cite{SainudiinVeber2016}.

\begin{model}[A biparametric Beta-splitting model]\label{M:BSM}
Let $(B_1,B_2,\ldots)$ be a sequence of independent and identically distributed (i.i.d.) random variables, with the $\calB(\alpha+1,\beta+1)$ distribution. Let also $(U_1,U_2,\ldots)$ be a sequence of i.i.d. random variables with the uniform distribution on $[0,1]$, that is independent of $(B_1,B_2,\ldots)$.

Let $((u_i,b_i))_{i\in \N}$ be a realization of $((U_i,B_i))_{i\in \N}$.  
The tree construction proceeds incrementally as follows, until the tree created has $n$ leaves.  
We start with a single root node, labelled by the interval $[0,1]$.
\begin{itemize}
\item Step $1$: Split the root into a left leaf labelled by $[0,b_1]$ and a right leaf labelled by $[b_1,1]$. Change the label of the root to the integer $1$.
\item Step $2$: If $u_2\in [0,b_1]$, split the left child node of the root into a left leaf and a right leaf respectively labelled by $[0,b_1b_2]$ and $[b_1b_2,b_1]$. If $u_2\in [b_1,1]$, then instead split the right child node of the root into left and right leaves with respective labels $[b_1,b_1+(1-b_1)b_2]$, $[b_1+(1-b_1)b_2,1]$. Label the former leaf that is split during this step by $2$.
\item Step $i$: Find the leaf whose interval label $[a,b]$ contains $u_i$. Change its label to the integer $i$ and split it into a left leaf with label $[a,a+(b-a)b_i]$ and a right leaf with label $[a+(b-a)b_i,b]$.
\item Stop at the end of Step $n-1$.
\end{itemize}

In words, at each step $i$ the labels of the leaves form a partition of the interval $[0,1]$. We find the next leaf to be split by checking which interval contains the corresponding $u_i$ and then $b_i$ is used to split the interval of that former leaf, say with length $\ell$, into two intervals of lengths $b_i\ell$ and $(1-b_i)\ell$. The internal node just created is then labelled by $i$ to record the order of the splits. At the end of step $i$, the tree has $i+1$ leaves, and so we stop the procedure at step $n-1$. Figure~\ref{example1} shows an example of such construction for $n=4$.  The probability of obtaining a ranked plane tree $\rpt{t}_n$ under this model by erasing the interval leaf labels is given in \cite[Thm.~1]{SainudiinVeber2016} by integrating over all possible splits.  
An interpretation of this model for transmission trees in terms of the underlying contact network of hosts undergoing an epidemic is given in \cite{SainudiinWelch2016}.
\comment{
\begin{align}
\mathbb{P}(\tau) &= \prod_{i=1}^{n-1} \left\{\frac{1}{B(\alpha+1,\beta+1)}\int_0^1 b_i^{n_i^L+\alpha}(1-b_i)^{n_i^R+\beta}db_i\right\} \notag\\
& = \prod_{i=1}^{n-1} \frac{B(n_i^L+\alpha+1,n_i^R+\beta+1)}{B(\alpha+1,\beta+1)}, \label{proba}
\end{align}
where $B(\alpha,\beta)$
}
\end{model}

\begin{figure}[t]
\begin{center}
\setlength{\unitlength}{908sp}%
\begingroup\makeatletter\ifx\SetFigFont\undefined%
\gdef\SetFigFont#1#2#3#4#5{%
  \reset@font\fontsize{#1}{#2pt}%
  \fontfamily{#3}\fontseries{#4}\fontshape{#5}%
  \selectfont}%
\fi\endgroup%
\begin{picture}(16135,8908)(-6539,-8450)
{\color[rgb]{0,0,0}\thinlines
\put(-3286,-3136){\circle*{424}}
}%
{\color[rgb]{0,0,0}\put(-2086,-4636){\circle*{424}}
}%
{\color[rgb]{0,0,0}\put(-4486,-4636){\circle*{424}}
}%
{\color[rgb]{0,0,0}\put(-5686,-6136){\circle*{424}}
}%
{\color[rgb]{0,0,0}\put(-3286,-6136){\circle*{424}}
}%
{\color[rgb]{0,0,0}\put(-4499,-7636){\circle*{424}}
}%
{\color[rgb]{0,0,0}\put(-2099,-7636){\circle*{424}}
}%
\thicklines
{\color[rgb]{0,0,0}\put(-3328,-3100){\line( 4,-5){1200}}
}%
{\color[rgb]{0,0,0}\put(-3279,-3092){\line(-4,-5){1200}}
}%
{\color[rgb]{0,0,0}\put(-4461,-4607){\line(-4,-5){1200}}
}%
{\color[rgb]{0,0,0}\put(-4435,-4660){\line( 4,-5){1200}}
}%
{\color[rgb]{0,0,0}\put(-3262,-6182){\line( 4,-5){1200}}
}%
{\color[rgb]{0,0,0}\put(-3315,-6138){\line(-4,-5){1200}}
}%
{\color[rgb]{0,0,0}\thinlines
\put(6001,239){\circle*{424}}
}%
{\color[rgb]{0,0,0}\put(4801,-1261){\circle*{424}}
}%
{\color[rgb]{0,0,0}\put(7201,-1261){\circle*{424}}
}%
{\color[rgb]{0,0,0}\put(-2099,-586){\circle*{424}}
}%
{\color[rgb]{0,0,0}\put(8176,-4411){\circle*{424}}
}%
{\color[rgb]{0,0,0}\put(9376,-5911){\circle*{424}}
}%
{\color[rgb]{0,0,0}\put(6976,-5911){\circle*{424}}
}%
{\color[rgb]{0,0,0}\put(8176,-7411){\circle*{424}}
}%
{\color[rgb]{0,0,0}\put(5776,-7411){\circle*{424}}
}%
\thicklines
{\color[rgb]{0,0,0}\put(6001,239){\line( 4,-5){1200}}
}%
{\color[rgb]{0,0,0}\put(6001,239){\line(-4,-5){1200}}
}%
{\color[rgb]{0,0,0}\put(8205,-4374){\line(-4,-5){1200}}
}%
{\color[rgb]{0,0,0}\put(8185,-4404){\line( 4,-5){1200}}
}%
{\color[rgb]{0,0,0}\put(7024,-5889){\line(-4,-5){1200}}
}%
{\color[rgb]{0,0,0}\put(6985,-5904){\line( 4,-5){1200}}
}%
{\color[rgb]{0,0,0}\put(6751,-2461){\vector( 0,-1){1500}}
}%
{\color[rgb]{0,0,0}\put(6316,-5311){\vector(-1, 0){6690}}
}%
{\color[rgb]{0,0,0}\put(-1049,-586){\vector( 1, 0){4619}}
}%
\put(-2849,-3286){\makebox(0,0)[lb]{\smash{{\SetFigFont{9}{10.8}{\rmdefault}{\mddefault}{\updefault}{\color[rgb]{0,0,0}$1$}%
}}}}
\put(-2849,-6361){\makebox(0,0)[lb]{\smash{{\SetFigFont{9}{10.8}{\rmdefault}{\mddefault}{\updefault}{\color[rgb]{0,0,0}$3$}%
}}}}
\put(-5324,-4861){\makebox(0,0)[lb]{\smash{{\SetFigFont{9}{10.8}{\rmdefault}{\mddefault}{\updefault}{\color[rgb]{0,0,0}$2$}%
}}}}
\put(-6524,-6961){\makebox(0,0)[lb]{\smash{{\SetFigFont{9}{10.8}{\rmdefault}{\mddefault}{\updefault}{\color[rgb]{0,0,0}$[0,b_1b_2]$}%
}}}}
\put(-2399,-8386){\makebox(0,0)[lb]{\smash{{\SetFigFont{9}{10.8}{\rmdefault}{\mddefault}{\updefault}{\color[rgb]{0,0,0}$[b',b_1]$}%
}}}}
\put(-5549,-8386){\makebox(0,0)[lb]{\smash{{\SetFigFont{9}{10.8}{\rmdefault}{\mddefault}{\updefault}{\color[rgb]{0,0,0}$[b_1b_2,b']$}%
}}}}
\put(-2474,-5386){\makebox(0,0)[lb]{\smash{{\SetFigFont{9}{10.8}{\rmdefault}{\mddefault}{\updefault}{\color[rgb]{0,0,0}$[b_1,1]$}%
}}}}
\put(6451, 89){\makebox(0,0)[lb]{\smash{{\SetFigFont{9}{10.8}{\rmdefault}{\mddefault}{\updefault}{\color[rgb]{0,0,0}$1$}%
}}}}
\put(8626,-4561){\makebox(0,0)[lb]{\smash{{\SetFigFont{9}{10.8}{\rmdefault}{\mddefault}{\updefault}{\color[rgb]{0,0,0}$1$}%
}}}}
\put(6151,-6136){\makebox(0,0)[lb]{\smash{{\SetFigFont{9}{10.8}{\rmdefault}{\mddefault}{\updefault}{\color[rgb]{0,0,0}$2$}%
}}}}
\put(-2474,-1336){\makebox(0,0)[lb]{\smash{{\SetFigFont{9}{10.8}{\rmdefault}{\mddefault}{\updefault}{\color[rgb]{0,0,0}$[0,1]$}%
}}}}
\put(4351,-2011){\makebox(0,0)[lb]{\smash{{\SetFigFont{9}{10.8}{\rmdefault}{\mddefault}{\updefault}{\color[rgb]{0,0,0}$[0,b_1]$}%
}}}}
\put(6976,-2011){\makebox(0,0)[lb]{\smash{{\SetFigFont{9}{10.8}{\rmdefault}{\mddefault}{\updefault}{\color[rgb]{0,0,0}$[b_1,1]$}%
}}}}
\put(7801,-8161){\makebox(0,0)[lb]{\smash{{\SetFigFont{9}{10.8}{\rmdefault}{\mddefault}{\updefault}{\color[rgb]{0,0,0}$[b_1b_2,b_1]$}%
}}}}
\put(5026,-8161){\makebox(0,0)[lb]{\smash{{\SetFigFont{9}{10.8}{\rmdefault}{\mddefault}{\updefault}{\color[rgb]{0,0,0}$[0,b_1b_2]$}%
}}}}
\put(-1649,-6736){\makebox(0,0)[lb]{\smash{{\SetFigFont{9}{10.8}{\rmdefault}{\mddefault}{\updefault}{\color[rgb]{0,0,0}$b'=b_1b_2+b_3b_1(1-b_2)$}%
}}}}
\put(7126,-3211){\makebox(0,0)[lb]{\smash{{\SetFigFont{9}{10.8}{\rmdefault}{\mddefault}{\updefault}{\color[rgb]{0,0,0}$0<u_2<b_1$}%
}}}}
\put(9301,-6736){\makebox(0,0)[lb]{\smash{{\SetFigFont{9}{10.8}{\rmdefault}{\mddefault}{\updefault}{\color[rgb]{0,0,0}$[b_1,1]$}%
}}}}
\put(-74,-286){\makebox(0,0)[lb]{\smash{{\SetFigFont{9}{10.8}{\rmdefault}{\mddefault}{\updefault}{\color[rgb]{0,0,0}$0<u_1<1$}%
}}}}
\put(1801,-5011){\makebox(0,0)[lb]{\smash{{\SetFigFont{9}{10.8}{\rmdefault}{\mddefault}{\updefault}{\color[rgb]{0,0,0}$b_1b_2<u_3<b_1$}%
}}}}
\end{picture}%
\caption{\label{example1}An example of a Beta-splitting tree construction for $k=3$.}
\end{center}
\end{figure}
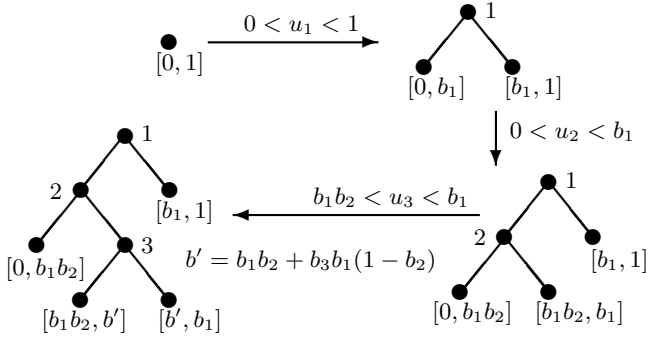

\begin{remark}[beta-splitting distribution over permutations]
Due to the bijection, $\rpT_n \ni \rpt{t}_n \leftrightarrow \sigma \in \ps_{n-1}$ via the increasing binary tree-lifting (Construction~\ref{R:TreeLifting}), 
we can transform samples from such $(\alpha,\beta)$-specified distributions over $\rpT_n$ to those over $\ps_{n-1}$, the $(n-1)!$ permutations of $[n-1]$. An interpretable biparametric family of distributions over permutations is naturally obtained by lifting the beta-splitting trees for each $(\alpha,\beta) \in (-1,\infty)\times (-1,\infty)$. 
This induced biparametric family of distributions over permutations can in turn be used to study possibly new properties of various randomized algorithms (including various sorting algorithms) that typically assume the input distribution to be the uniform distribution over permutations, i.e., the special case of $(\alpha,\beta) = (0,0)$ in the family.
\end{remark}

\subsection{Two properties and their consequences}

\begin{definition}[split-exchangeable model]
If $\Pr\{X_{0:k}=x_{0:k} \}$ or $\Pr \{ \rpt{T}_k = \rpt{t}_k \}$ obtained from a Markov splitting model is identical for every $x_{0:k}$ or $\rpt{t}_k$ that has the same partition $x_{(0:k)}$ or the same planar tree $\pt{t}_k$, then the model is said to be {\em split-exchangeable}.  
Note that Models~\ref{M:US}, \ref{M:IfS} and \ref{M:BSM} are split-exchangeable while Model~\ref{M:DPS} is not.
\end{definition}

\begin{theorem}[Split-exchangeable planar tree probability]\label{T:SE}
The probability of reaching a given partition or plane binary tree under a Markov splitting model that satisfies split-exchangeability is:
\[
\Pr \{ X_{(0:k)}  = x_{(0:k)}\} = B(x_{(0:k)}) \times \Pr\{x_{0:k}\} . 
\] 
or equivalently in tree notation:
\[
\Pr \{ \pt{T}_k = \pt{t}_k \} = B(\pt{t}_k) \times \Pr\{ \rpt{t}_k \}
\]
where $\rpt{t}_k \in \{ \rpt{t}_k: \rpt{t}_k \mapsto \pt{t}_k \}$ and $B$ is the Catalan coefficient.
\begin{proof}
Since the probabilities are split-exchangeable, $\Pr(\rpt{t})$ is identical for each $\rpt{t}_k \in \{ \rpt{t}_k: \rpt{t}_k \mapsto \pt{t}_k \}$.  
\[
\Pr \{ \pt{T}_k = \pt{t}_k \} 
= \Pr \{ \rpt{t}_k: \rpt{t}_k \mapsto \pt{t}_k\} 
= \#\{ \rpt{t}_k: \rpt{t}_k \mapsto \pt{t}_k\} \times \Pr\{ \rpt{t}_k \}
= B(\pt{t}_k) \times \Pr\{ \rpt{t}_k \}
\]
The last equality is due to \eqref{E:CatCoeff}.
\end{proof}
\end{theorem}

\begin{definition}[plane-invariant model]
If $\Pr \{ \pt{T}_k = \pt{t}_k \}$ is identical for every $\pt{t}_k$ that has the same tree $t_k$ then the model is said to be {\em plane-invariant}.
For example, Model~\ref{M:US} is plane-invariant but Model~\ref{M:IfS} is not.
\end{definition}

\begin{theorem}[Split-exchangeable plane-invariant tree probability]\label{T:SEPI}
The probability of a tree under a Markov splitting model that satisfies split-exchangeability and plane-invariance is:
\[
\Pr \{ T_k = t_k \} = B(t_k) \times 2^{k-1-s(t_k)} \times \Pr\{ \rpt{t}_k \}
\]
where $\rpt{t}_k \in \{ \rpt{t}_k: \rpt{t}_k \mapsto \pt{t}_k \}$, and $\pt{t}_k \in \{\pt{t}_k: \pt{t}_k \mapsto t_k \}$, $B$ is the Catalan coefficient and $s(t_k)$ is the number of symmetry nodes in $t_k$.
\begin{proof}
Since the probabilities are split-exchangeable, $\Pr(\rpt{t})$ is identical for each $\rpt{t}_k \in \{ \rpt{t}_k: \rpt{t}_k \mapsto \pt{t}_k \}$. 
And, since the probabilities are plane-invariant, $\Pr(\pt{t})$ is identical for each $\pt{t}_k \in \{ \pt{t}_k: \pt{t}_k \mapsto t_k \}$. 
Therefore,
\[
\Pr \{ T_k = t_k \} 
= \Pr \{ \rpt{t}_k: \rpt{t}_k \mapsto t_k\} 
= B(\pt{t}_k) \times 2^{k-1-s(t_k)} \times \Pr\{ \rpt{t}_k \}
\]
The last equality is due to \eqref{E:NumRptPerT}.
\end{proof}
\end{theorem}




\subsection{Applications of split-exchangeability and plane-invariance}
 
Theorems \ref{T:SE} and \ref{T:SEPI} can be useful in obtaining probabilities of trees at coarser resolutions from the probabilities at the finer resolution of ranked planar trees if they satisfy split-exchangeability (and plane-invariance).  
For example, Model~\ref{M:US} for Yule trees with $\Pr(\rpt{t}_k)=1/(k-1)!$ for every $\rpt{t}_k \in \rpt{T}_k$, is split-exchangeable and plane-invariant and therefore by Theorem~\ref{T:SEPI},
\[
\Pr \{ T_k = t_k \} = \frac{(n-1)!}{\prod_{v \in \iV} \lfloor t_k(v) \rfloor} \times 2^{k-1-s(t_k)} \times \frac{1}{(k-1)!}.
\]
This gives the nonuniform probability of an unranked and nonplanar Yule tree in terms of the product of its subtree splits and symmetry nodes.

The Beta-splitting model, a biparametric generalization of the Yule model and several other speciation models in phylogenetics, is split-exchangeable for any $\alpha$ and $\beta$, but only plane-invariant when $\alpha=\beta$.  
These properties of the Beta-splitting model are used in \cite{SainudiinVeber2016} to readily obtain probabilities of planar trees, nonplanar ranked trees and nonplanar unranked trees from those of their corresponding ranked planar trees.

In nonparametic density estimation, where one has to reconstruct the unknown density from which $n$ data points have been sampled, a fundamental problem is to obtain data-adaptive partitions of the support set, say $[0,1]$ without loss of generality.  
Such partitioning schemes can be formulated as split-exchangeable Markov chains on $\rpt{T}_k$ by generalizing Model~\ref{M:IfS}, where 
the objects of interest are the partitions encoded by $\pt{T}_k$ with probabilities given by Theorem~\ref{T:SE}.

\subsection{Thompson's group $F$} 
We next describe how probability models on plane binary trees are used to obtain insights in geometric group theory involving Thompson's group.  
Thompson's group $F$ is a group with a range of unusual properties and a wide range of characterizations.
Here, we consider elements of Thompson's group $F$ as piecewise-linear orientation-preserving homeomorphisms of the unit interval $[0,1]$ to itself, with slopes that are powers of two, and with breakpoint sets
that are contained in the dyadic rationals.  Such elements can be described by pairs of rooted binary trees of the same size, where the corresponding group element is the piecewise-linear interpolation of the corresponding dyadic partitions described by the trees.  The widths of the intervals is exactly the dyadic width described in Representation \ref{depths}.  
For further background on Thompson's group $F$, see the introduction by Cannon, Floyd and Parry \cite{MR1426438}.  

\begin{figure}
\begin{center}
\includegraphics[width=.6 \textwidth]{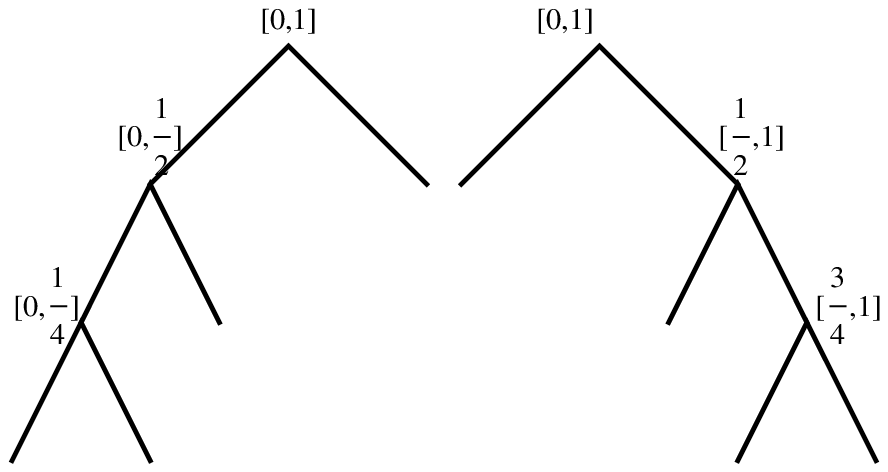}
\caption{A pair of trees giving partitions $(0,\frac14,\frac12,1)$ and $(0,\frac12,\frac34,1)$ respectively }
\label{fig:treepair}
\end{center}
\end{figure}

\begin{figure}
\begin{center}
\includegraphics[width=.6 \textwidth]{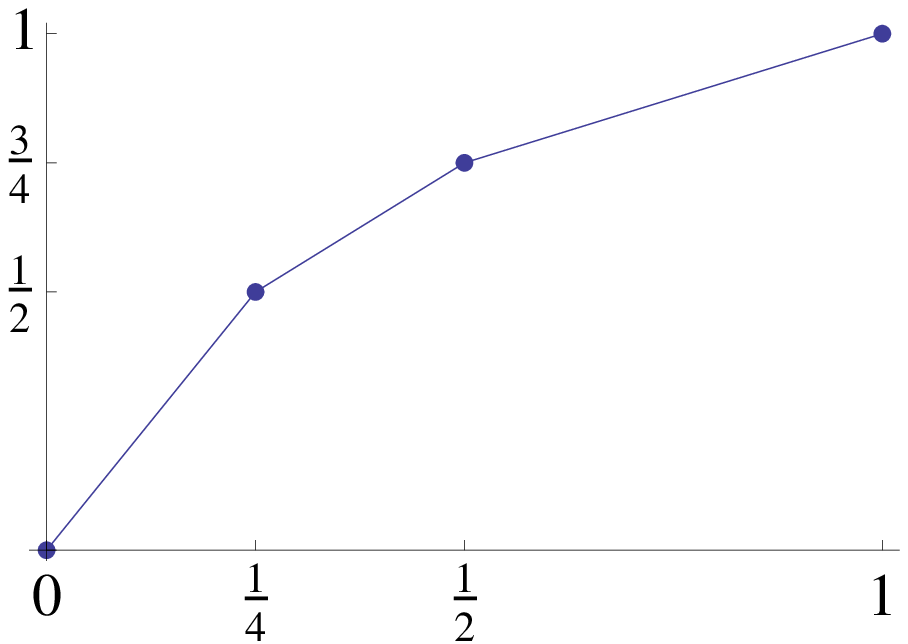}
\caption{Interpolating between the  two partitions of the unit interval  $\{0, \frac14, \frac12, 1\}$ and $\{0, \frac12, \frac34, 1\}$ to make a piecewise-linear homeomorphism}
\label{fig:interp}
\end{center}
\end{figure}

\subsubsection{Sampling in $F$}

Thompson's group $F$ is the simplest known example of a wide range of pathological group-theoretic behavior, with it serving as counterexamples to a wide range of conjectures.
Furthermore, there are a number of properties of $F$ which are not known despite a great deal of study over the last 40 years.  There have thus been a number of computational
experiments by a range of authors \cite{yulereduce, MR3101764, MR3035049, MR3043436,MR2646034,MR2566302,MR2395786} designed to yield insight into the group-theoretic
properties of $F$.  For those involving sampling, there are three methods that have been used to sample elements at random of increasingly large subsets of $F$.   The first is
to sample words chosen at random from the balls of size $n$ with respect to the standard word metric for $F$ (using the generating set $\{x_0,x_1\}$).  This has a number of
desired properties but unfortunately the sizes of the metric balls are not known generally, not even asymptotically, and though it is known that the growth rate is exponential, the
growth rate is unknown.  There are proven upper and lower bounds for the exponential growth rate \cite{MR2104775} and compelling computational evidence \cite{MR2646034} that
the growth rate is extremely close to the upper bound but not knowing the growth properties prevents the analysis for sampling to understand the asymptotic behavior from being feasible.
The second method of sampling \cite{MR2566302,MR3101764,MR3035049} 
analyzes sampling in $F$ by choosing tree pairs of size $n$ uniformly at random, performing the appropriate reductions, and considering the resulting group elements.  Here, the trees are selected uniformly at random from all trees of size $n$.  The third way \cite{yulereduce}, selects trees via a bifurcation process modeled on the Yule distribution \cite{Yule1924} for rooted binary trees.

\subsubsection{Distortion of sampling in $F$}

The Catalan coefficients described here exactly describe the sampling bias between the two methods of tree pair generation for unreduced tree pair diagrams.  That is, a given pair $(\pt{s}_n,\pt{t}_n) \in \pT_n^2$ has weight $1$
with respect to the uniform distribution on trees, and has chance of selection $1/C_n^2$, where $C_n$ is the $n$-th Catalan number.  The same tree pair has weight $B(\pt{s}_n) B(\pt{t}_n)$ where $B(\pt{t}_n)$ is the number of ways that tree $\pt{t}_n$ can arise via a bifurcation process, giving a chance of selection of  $B(\pt{s}_n) B(\pt{t}_n)/(n!)^2$.  We note that a number of authors have analyzed different properties of these two tree distributions in the unordered case-- see, for example, McKenzie and Steel \cite{cherries} where the distribution of the number of sibling pairs (or ``cherries'') is analyzed.


\subsection*{Acknowledgment}
RS~was partly supported by a Sabbatical Grant from College of Engineering, University of Canterbury, a Visiting Scholarship at Department of Mathematics, Cornell University, Ithaca, NY, USA, consulting revenues from Wynyard Group and by the chaire Mod\'elisation Math\'ematique et Biodiversit\'e of Veolia Environnement-\'Ecole Polytechnique-Museum National d'Histoire Naturelle-Fondation X. 
SC is grateful for NSF support through grant \#1417820 and to Simons Foundation for grant \#234548.

\bibliographystyle{plain}
\bibliography{references}

\begin{thebibliography}{10}

\bibitem{Aldous2001}
David~J. Aldous.
\newblock Stochastic models and descriptive statistics for phylogenetic trees,
  from {Y}ule to today.
\newblock {\em Statist. Sci.}, 16(1):23--34, 2001.

\bibitem{Blum2006}
Michael G.~B. Blum and Olivier Fran\c{c}ois.
\newblock {Which random processes describe the tree of life? A large-scale
  study of phylogenetic tree imbalance}.
\newblock {\em Systematic Biology}, 55(4):685--691, 2006.

\bibitem{MR2395786}
Jos{\'e} Burillo, Sean Cleary, and Bert Wiest.
\newblock Computational explorations in {T}hompson's group {$F$}.
\newblock In {\em Geometric group theory}, Trends Math., pages 21--35.
  Birkh\"auser, Basel, 2007.

\bibitem{MR1426438}
J.~W. Cannon, W.~J. Floyd, and W.~R. Parry.
\newblock Introductory notes on {R}ichard {T}hompson's groups.
\newblock {\em Enseign. Math. (2)}, 42(3-4):215--256, 1996.

\bibitem{MR3101764}
Timothy Chu and Sean Cleary.
\newblock Expected conflicts in pairs of rooted binary trees.
\newblock {\em Involve}, 6(3):323--332, 2013.

\bibitem{MR2566302}
Sean Cleary, Murray Elder, Andrew Rechnitzer, and Jennifer Taback.
\newblock Random subgroups of {T}hompson's group {$F$}.
\newblock {\em Groups Geom. Dyn.}, 4(1):91--126, 2010.

\bibitem{yulereduce}
Sean Cleary, John Passaro, and Yasser Toruno.
\newblock Average reductions in yule-generated binary trees.
\newblock "to appear".

\bibitem{MR3035049}
Sean Cleary, Andrew Rechnitzer, and Thomas Wong.
\newblock Common edges in rooted trees and polygonal triangulations.
\newblock {\em Electron. J. Combin.}, 20(1):Paper 39, 22, 2013.

\bibitem{Dobrow1995}
Robert~P. Dobrow and James~Allen Fill.
\newblock On the markov chain for the move-to-root rule for binary search
  trees.
\newblock {\em The Annals of Applied Probability}, 5(1):1--19, 02 1995.

\bibitem{Donaghey1975}
Robert Donaghey.
\newblock Alternating permutations and binary increasing trees.
\newblock {\em Journal of Combinatorial Theory, Series A}, 18(2):141 -- 148,
  1975.

\bibitem{MR2646034}
Murray Elder, {\'E}ric Fusy, and Andrew Rechnitzer.
\newblock Counting elements and geodesics in {T}hompson's group {$F$}.
\newblock {\em J. Algebra}, 324(1):102--121, 2010.

\bibitem{MR3043436}
Murray Elder, Andrew Rechnitzer, and Thomas Wong.
\newblock On the cogrowth of {T}hompson's group {$F$}.
\newblock {\em Groups Complex. Cryptol.}, 4(2):301--320, 2012.

\bibitem{Fill1996}
James~Allen Fill.
\newblock On the distribution of binary search trees under the random
  permutation model.
\newblock {\em Random Structures \& Algorithms}, 8(1):1--25, 1996.

\bibitem{Flajolet}
Philippe Flajolet and Robert Sedgewick.
\newblock {\em Analytic Combinatorics}.
\newblock Cambridge University Press, New York, NY, USA, 1 edition, 2009.

\bibitem{Foata1971}
D.~Foata and M.P. Sch\"utzenberger.
\newblock Nombres d'euler et permutations alternantes (unabriged version, 71
  pages).
\newblock Department of mathematics, University of Florida, Gainesville,
  available from http://www.mat.univie.ac.at/∼slc/, 1971.

\bibitem{Gessaman1970}
M.~P. Gessaman.
\newblock A consistent nonparametric multivariate density estimator based on
  statistically equivalent blocks.
\newblock {\em The Annals of Mathematical Statistics}, 41(4):pp. 1344--1346,
  1970.

\bibitem{MR2104775}
V.~S. Guba.
\newblock On the properties of the {C}ayley graph of {R}ichard {T}hompson's
  group {$F$}.
\newblock {\em Internat. J. Algebra Comput.}, 14(5-6):677--702, 2004.
\newblock International Conference on Semigroups and Groups in honor of the
  65th birthday of Prof. John Rhodes.

\bibitem{Kirkpatrick1993}
Mark Kirkpatrick and Montgomery Slatkin.
\newblock {Searching for Evolutionary Patterns in the Shape of a Phylogenetic
  Tree}.
\newblock {\em Evolution}, 47(4):1171--1181, 1993.

\bibitem{Mahmoud1992}
Hosam~M. Mahmoud.
\newblock {\em Evolution of random search trees}.
\newblock Wiley-Interscience series in discrete mathematics and optimization.
  Wiley, 1992.

\bibitem{Catalan2005}
G.~McGarvey and B.~Cloitre.
\newblock Sequence {A}000108, {T}he {O}n-line {E}ncyclopedia of {I}nteger
  {S}equences.
\newblock published electronically, Feb 2005.

\bibitem{cherries}
Andy McKenzie and Mike Steel.
\newblock Distributions of cherries for two models of trees.
\newblock {\em Math. Biosci.}, 164(1):81--92, 2000.

\bibitem{Meier2008}
J.~Meier.
\newblock {\em Groups, graphs and trees: an introduction to the geometry of
  infinite groups}.
\newblock London Mathematical Society student texts. Cambridge University
  Press, Cambridge, 2008.

\bibitem{Murtagh1984}
Fionn Murtagh.
\newblock Counting dendrograms: a survey.
\newblock {\em Discrete Applied Mathematics}, 7(2):191--199, 1984.

\bibitem{CatalanCoeff2012}
R.~Sainudiin.
\newblock Sequence {A}185155, {T}he {O}n-line {E}ncyclopedia of {I}nteger
  {S}equences.
\newblock published electronically, Feb 2012.

\bibitem{Sainudiin2015}
Raazesh Sainudiin, Tanja Stadler, and Amandine V{\'e}ber.
\newblock Finding the best resolution for the kingman--tajima coalescent:
  theory and applications.
\newblock {\em Journal of Mathematical Biology}, 70(6):1207--1247, 2015.

\bibitem{SainudiinVeber2016}
Raazesh Sainudiin and Amandine V{\'e}ber.
\newblock A beta-splitting model for evolutionary trees.
\newblock {\em Royal Society Open Science}, 3(5), 2016.

\bibitem{SainudiinWelch2016}
Raazesh Sainudiin and David Welch.
\newblock The transmission process: A combinatorial stochastic process for the
  evolution of transmission trees over networks.
\newblock {\em Journal of Theoretical Biology}, 410:137 -- 170, 2016.

\bibitem{Steel2003}
Charles Semple and Mike Steel.
\newblock {\em Phylogenetics}, volume~24 of {\em Oxford Lecture Series in
  Mathematics and its Applications}.
\newblock Oxford University Press, Oxford, 2003.

\bibitem{A001190}
N.J.A. Sloane.
\newblock Sequence {A}001190, {T}he {O}n-line {E}ncyclopedia of {I}nteger
  {S}equences.
\newblock published electronically, May 1995.

\bibitem{MaxHeaps2007}
N.J.A. Sloane.
\newblock Sequence {A}056971, {T}he {O}n-line {E}ncyclopedia of {I}nteger
  {S}equences.
\newblock published electronically, Nov 2007.

\bibitem{EulerNumbers2013}
N.J.A. Sloane.
\newblock Sequence {A}000111, {T}he {O}n-line {E}ncyclopedia of {I}nteger
  {S}equences.
\newblock published electronically, May 2013.

\bibitem{Stanley1997}
Richard~P. Stanley.
\newblock {\em Enumerative combinatorics. {V}ol. 1}, volume~49 of {\em
  Cambridge Studies in Advanced Mathematics}.
\newblock Cambridge University Press, Cambridge, 1997.
\newblock With a foreword by Gian-Carlo Rota, Corrected reprint of the 1986
  original.

\bibitem{Stanley1999}
Richard~P. Stanley.
\newblock {\em Enumerative combinatorics. {V}ol. 2}, volume~62 of {\em
  Cambridge Studies in Advanced Mathematics}.
\newblock Cambridge University Press, Cambridge, 1999.

\bibitem{Tajima1983}
F.~Tajima.
\newblock Evolutionary relationship of {DNA} sequences in finite populations.
\newblock {\em Genetics}, 105:437--460, 1983.

\bibitem{Yule1924}
G.~U. Yule.
\newblock A mathematical theory of evolution: based on the conclusions of {D}r.
  {J}.{C}. {W}illis.
\newblock {\em Philos. Trans. Roy. Soc. London Ser. B}, 213:21--87, 1924.

\end{thebibliography}


\end{document}